\documentclass[10pt]{article}
\usepackage{amssymb}
\usepackage{graphicx}
\usepackage{xcolor} % A package to add color.
\usepackage{tensor}
\usepackage{fullpage} % Sets all margins to 1 in.  
\usepackage{amsmath}
\usepackage{amsthm}
\usepackage{verbatim}
\usepackage{txfonts,bm,mathtools}

\usepackage{enumitem}
\setlist[enumerate]{leftmargin=1.5em}
\setlist[itemize]{leftmargin=1.5em}

\usepackage{hyperref}
\usepackage{seqsplit,mathtools} 
\usepackage{caption}
\usepackage{subcaption}
\usepackage{ulem}

\mathtoolsset{showonlyrefs}

\definecolor{green}{rgb}{0,0.5,0} % Redefines the color green.
\newcommand{\Red}[1]{\begingroup\color{red} #1\endgroup} % Red
\newcommand{\Blue}[1]{\begingroup\color{blue} #1\endgroup} % Red
\newcommand{\Green}[1]{\begingroup\color{green} #1\endgroup}

\newtheorem{thm}{Theorem}[section]

\newtheorem{lem}[thm]{Lemma}
\newtheorem{prop}[thm]{Proposition}

\newtheorem{defn}[thm]{Definition}

\theoremstyle{definition}

\theoremstyle{remark}
\newtheorem{rmk}[thm]{Remark}
\newtheorem*{note}{Note}

\numberwithin{equation}{section}
\newcommand{\nrm}[1]{\Vert#1\Vert}

\newcommand{\tld}[1]{\widetilde{#1}}
\newcommand{\br}[1]{\overline{#1}}

\newcommand{\nnrm}[1]{{\vert\kern-0.25ex\vert\kern-0.25ex\vert #1 
		\vert\kern-0.25ex\vert\kern-0.25ex\vert}}

\newcommand{\lap}{\Delta}

\newcommand{\rd}{\partial}
\newcommand{\nb}{\nabla}

\newcommand{\ift}{\infty}

\newcommand{\alp}{\alpha}

\newcommand{\gmm}{\gamma}

\newcommand{\dlt}{\delta}

\newcommand{\eps}{\epsilon}
\newcommand{\veps}{\varepsilon}
\newcommand{\kpp}{\kappa}
\newcommand{\lmb}{\lambda}

\newcommand{\Sgm}{\Sigma}
\newcommand{\tht}{\theta}

\newcommand{\vtht}{\vartheta}
\newcommand{\omg}{\omega}
\newcommand{\Omg}{\Omega}

\newcommand{\bbN}{\mathbb N}

\newcommand{\bbR}{\mathbb R}
\newcommand{\bbS}{\mathbb S}

\begin{document}
	
	\bibliographystyle{plain}
	\title{On global regularity of some bi-rotational Euler flows in $\bbR^{4}$}
	\author{Kyudong Choi\thanks{Department of Mathematical Sciences, Ulsan National Institute of Science and Technology, 50 UNIST-gil, Eonyang-eup, Ulju-gun, Ulsan 44919, Republic of Korea. Email: kchoi@unist.ac.kr}\and In-Jee Jeong\thanks{Department of Mathematical Sciences and RIM, Seoul National University, 1 Gwanak-ro, Gwanak-gu, Seoul 08826, Republic of Korea. Email: injee$ \_ $j@snu.ac.kr}\and 
		Deokwoo Lim\thanks{The Research Institute of Basic Sciences, Seoul National University, 1 Gwanak-ro, Gwanak-gu, Seoul 08826, Republic of Korea. Email: dwlim95@snu.ac.kr} }
	
	\date\today
	\maketitle
	
	\begin{abstract}
		In this paper, we consider %local and global well-posedness of 
		incompressible Euler flows in $ \bbR^{4} $ under bi-rotational symmetry, namely solutions that are invariant under rotations in $\bbR^{4}$ fixing either the first two or last two axes. %satisfying some conditions around axes of rotation. 
		With the additional swirl-free assumption, our first main result gives local wellposedness of Yudovich-type solutions, extending the work of 
		Danchin [Uspekhi Mat. Nauk 62(2007), no.3, 73–94] for axisymmetric flows in $\bbR^{3}$. 
		The second main result establishes global wellposedness under additional decay conditions near the axes and at infinity. This in particular gives global regularity of $C^{\infty}$ smooth and decaying Euler flows in $\bbR^{4}$ subject to bi-rotational symmetry without swirl.
	\end{abstract}
	
	\renewcommand{\thefootnote}{\fnsymbol{footnote}}
	\renewcommand{\thefootnote}{\arabic{footnote}}
	
	\section{Introduction}\label{sec_intro}
	
	%\subsection{The bi-rotational incompressible Euler equations}
	
	We consider the Cauchy problem for the incompressible Euler equations in $\bbR^n$: 
	\begin{equation}\label{eq_Eulereq}
		\left\{
		\begin{aligned}
			\rd_{t}u+(u\cdot\nb)u+\nb p&=0,\\
			\nb\cdot u&=0,\\
			u|_{t=0}&=u_{0},
		\end{aligned}
		\right.
	\end{equation}
	where $ u=(u^{1},\cdots,u^{n}) : [0,\ift)\times\bbR^{n}\to\bbR^{n} $ and $ p : [0,\ift)\times\bbR^{n}\to\bbR $. The system \eqref{eq_Eulereq} describes the  motion of volume-preserving and inviscid fluid flows, and has been studied extensively in the two- and three-dimensional cases, namely when $n = 2, 3$. For $n \ge 3$, global regularity or finite time singularity for $C^{\infty}$ smooth and decaying solutions to \eqref{eq_Eulereq} is an outstanding open problem. 
	
	The goal of this paper is to study the four-dimensional case $(n=4)$ under the \textit{bi-rotational} symmetry. The main motivation is to uncover new geometric scenarios that could contribute to the understanding of global regularity or the formation of finite-time singularities for the Euler equations.
	
	\subsection{The system for bi-rotationally symmetric solutions}\label{subsec_birotsys}

	The three-dimensional Euler equations have been extensively studied under the axisymmetry assumption, which means that the solution is invariant under all rotations fixing one axis in $\bbR^{3}$. Since the Euler equations preserve rotational symmetries of the solution, it suffices to assume axisymmetry for the initial data. The key simplification comes from the fact that the number of independent spatial variables reduces from three to two. 
	
	While higher dimensional Euler equations ($n\ge4$) has been studied under the axisymmetry assumption as well, in this paper we shall investigate the bi-rotational symmetry, or $O(2)\times O(2)$ symmetry, which is specific to dimensions 4 or higher. As the name suggests, this symmetry refers to solutions which are invariant under all rotations fixing either the first two or the last two variables. To handle such solutions, it is natural to consider the bi-polar coordinates in $\bbR^4$: the coordinate system $ (r,\tht,s,\phi) $ with $x_1= r\cos(\tht), \quad x_2 = r\sin(\tht), \quad x_3 = s\cos(\tht)$ and $x_4 = s\sin(\phi).$ 
	
	Under the bi-rotational symmetry, the velocity vector can be written as \begin{equation}\label{eq_bi-rotational}
		\begin{split}
			u = u^{r} e^{r} + u^{\tht} e^{\tht} + u^{s}e^{s} + u^{\phi} e^{\phi},
		\end{split}
	\end{equation} where the functions $u^{*}$ depend only on $(t,r,s)$ and the unit vectors $e^{*}$ are defined by \begin{equation*}
		\begin{split}
			e^{r} = \begin{pmatrix}
				\cos\tht \\
				\sin\tht \\
				0 \\
				0
			\end{pmatrix}, \quad  e^{\tht} = \begin{pmatrix}
				-\sin\tht\\
				\cos\tht\\
				0 \\
				0
			\end{pmatrix}, \quad e^{s} = \begin{pmatrix}
				0\\
				0\\
				\cos\phi\\
				\sin\phi 
			\end{pmatrix}, \quad e^{\phi} = \begin{pmatrix}
				0 \\
				0\\ 
				-\sin\phi \\
				\cos\phi
			\end{pmatrix} .
		\end{split}
	\end{equation*} Similarly, the pressure is given by $p = q(t,r,s)$ for some scalar-valued function $q$. We then compute that \eqref{eq_Eulereq} turns into \begin{equation}  \label{eq:4D-Euler-polar-radial}
		\left\{
		\begin{aligned} 
			&\rd_t u^{r} + (u^{r} \rd_{r} +u^{s} \rd_{s}  ) u^{r} = \frac{(u^{\tht})^2}{r} - \rd_{r}q \\
			&\rd_t u^{s} + (u^{r} \rd_{r} +u^{s} \rd_{s}  ) u^{s} = \frac{(u^{\phi})^2}{s} - \rd_{s}q 
		\end{aligned}
		\right.
	\end{equation} and \begin{equation}\label{eq:4D-Euler-polar-angular}
		\left\{
		\begin{aligned}
			& \rd_t u^{\tht} + (u^{r} \rd_{r} +u^{s} \rd_{s}  ) u^{\tht} = - \frac{u^{r}}{r}u^{\tht}  \\
			& \rd_t u^{\phi} +  (u^{r} \rd_{r} +u^{s} \rd_{s}  ) u^{\phi} = - \frac{u^{s}}{s} u^{\phi}
		\end{aligned}
		\right. 
	\end{equation} supplemented with the divergence-free condition \begin{equation}\label{eq:4D-div}
		\begin{split}
			\rd_{r}u^{r} + \frac{u^{r}}{r} +\rd_{s}u^{s} + \frac{u^{s}}{s} = 0.  
		\end{split}
	\end{equation} The system \eqref{eq:4D-Euler-polar-radial}--\eqref{eq:4D-div} describes the 4D Euler equations under the bi-rotational symmetry. We shall mainly work with its vorticity form (to avoid dealing with the pressure), obtained by introducing the scalar ``vorticity'' \begin{equation}\label{eq:4D-vort}
		\begin{split}
			w := \rd_{r}u^{s} - \rd_{s}u^{r}.
		\end{split}
	\end{equation} From \eqref{eq:4D-Euler-polar-radial}, we have that $w$ satisfies \begin{equation}\label{eq:4D-swirl}
		\begin{split}
			\rd_t w +  (u^{r} \rd_{r} +u^{s} \rd_{s}  ) w =  \left( \frac{u^{r}}{r} + \frac{u^{s}}{s}  \right) w + 2 \left( \frac{u^{\phi}}{s} \rd_{r}u^{\phi} - \frac{u^{\tht}}{r}\rd_{s}u^{\tht} \right). 
		\end{split}
	\end{equation} We may introduce $\psi$ satisfying \begin{equation}\label{eq:4D-psi}
		\begin{split}
			w = \rd_{r} \left( \frac{1}{r} \rd_{r}(r\psi) \right)+ \rd_{s} \left( \frac{1}{s} \rd_{s}(s\psi) \right)
		\end{split}
	\end{equation} subject to the boundary condition $\psi(0,\cdot) = \psi(\cdot,0) = 0$. The velocity $(u^{r},u^{s})$ is obtained from $\psi$ by \begin{equation}  \label{eq:4D-vel-stream} 
		\begin{aligned} 
			u^{r} &= -\frac{1}{s}\rd_{s}(s\psi) = - \frac{\psi}{s} - \rd_{s}\psi ,\quad u^{s} = \frac{1}{r}\rd_{r}(r\psi) = \frac{\psi}{r} + \rd_{r}\psi .
		\end{aligned} 
	\end{equation} %For simplicity, we shall refer to the system \eqref{eq:4D-Euler-polar-angular}, \eqref{eq:4D-swirl}--\eqref{eq:4D-vel-stream} as the bi-rotational Euler equations. %Since we shall be only concerned with velocities which are Lipschitz continuous in $\bbR^4$ (or at least log-Lipschitz to be defined below), the boundary conditions \begin{equation*} \begin{split}	u^{r}(0,s) = u^{\tht}(0,s) = 0, \quad u^{s}(r,0) = u^{\phi}(r,0) = 0		\end{split}	\end{equation*} need to be imposed. 
	%(To consider more regular velocities, additional vanishing conditions for the derivatives should be imposed; see \cite{Tam}.)
	
	%As it is frequently done for the axisymmetric Euler equations (cf. \cite[Section 5.4]{MB}), one can simplify the equation using advected quantities: with \begin{equation*}
		%	\begin{split}
			%		D_t = \rd_t + u^{r}\rd_{r}+u^{s}\rd_{s} , 
			%	\end{split}
		%\end{equation*} the system takes the form \begin{equation}\label{eq:4D-adv}
		%	\begin{split}
			%		D_t \left( \frac{w}{rs} \right) & = \frac{1}{rs^4} \rd_{r} \left[ (su^{\phi})^2 \right] - \frac{1}{r^4s} \rd_{s}\left[ (ru^{\tht})^2 \right], \\
			%		D_t   (ru^{\tht})^2   & = 0, \\
			%		D_t   (su^{\phi})^2 & = 0 .
			%	\end{split}
		%\end{equation} 
		The above system contains a few sub-systems; for instance, if one takes $u^{\phi} = 0$ (or similarly $u^{\tht}=0$), then it reduces to a system of two scalar quantities which resembles the 3D axisymmetric Euler equations. Moreover, if one takes 
		\begin{equation}\label{eq_noswirl}
			u^{\tht} = u^{\phi} = 0,
		\end{equation}
		then the system becomes \begin{equation}\label{eq:4D-adv-no-swirl}
			\begin{split}
				D_t \left( \frac{w}{rs} \right) & = 0, \qquad 	D_t = \rd_t + u^{r}\rd_{r}+u^{s}\rd_{s},
			\end{split}
		\end{equation} which shall be referred to as the bi-rotational Euler equation \textit{without swirl}. This is equivalent with
		\begin{equation}\label{eq_Vorticityform}
			\begin{split}
				\rd_{t}w+(u^{r}\rd_{r}+u^{s}\rd_{s})%u\cdot \nb 
				w&=w\bigg(\frac{u^{r}}{r}+\frac{u^{s}}{s}\bigg),
			\end{split}
		\end{equation} 
		which is simply \eqref{eq:4D-swirl} with the no-swirl condition $ u^{\tht}=u^{\phi}\equiv0 $.
		
		\subsection{Main results}\label{subsec_mainresults}
		
		In this paper, we prove local and global regularity results for the Cauchy problem to the bi-rotational Euler equations \eqref{eq_Vorticityform} without swirl. % \eqref{eq:4D-Euler-polar-angular}, \eqref{eq:4D-swirl}--\eqref{eq:4D-vel-stream}. 
		Before stating the theorems, let us provide a quick review on the well-posedness theory of the incompressible Euler equations. More refined statements which naturally motivate the results in this work will be given in the following section. 
		
		The incompressible Euler equations \eqref{eq_Eulereq} in $\bbR^n$ are locally well-posed in H\"older spaces $C^{k,\alpha}$ with $k\ge 1, 0<\alpha<1$ (together with some decay assumption at infinity) and Sobolev spaces $H^s$ with $s>\frac{n}{2}+1$, in terms of $u$. While not much is known in general about long-time behavior of the solutions, the Beale--Kato--Majda criterion  {(\cite{BKM,KaPo})} states that these local solutions lose regularity at some time $T^*>0$ if and only if \begin{equation*}
			\begin{split}
				\lim_{t\rightarrow T^*}\int_0^{t} \nrm{(\nabla\times u)(\tau)}_{L^\infty}d\tau = +\infty. 
			\end{split}
		\end{equation*} 
		
		In particular, the bi-rotational Euler equations (being a special case of the four-dimensional Euler equations) are locally well-posed in such function spaces. We are interested in utilizing the special nature of the bi-rotational system to obtain well-posedness (existence and uniqueness) in the largest possible \textit{borderline} spaces, which are characterized by having a scale-invariant norm: $\nrm{\cdot}$ satisfying $\nrm{(\nabla\times u)(\lmb \cdot)} = \nrm{\nabla\times u}$ for any $\lmb>0$. In this direction, we have a local well-posedness for Yudovich type initial data, extending earlier works of Ukhovsksii--Yudovich \cite{UY} and Danchin \cite{Danaxi} in the three-dimensional axisymmetric Euler: 
		\begin{thm}\label{thm_localreg}
			Assume that $w_{0}$ satisfy $w_{0} \in (L^{4,1}\cap L^\infty)(\bbR^4)$ and 
			$$ {(1+
				r+s)\frac{w_{0}}{rs}%\frac{w_{0}}{rs}%w_{0}/(rs) 
				\in L^{4,1}(\bbR^4).}$$ 
			%\Red{(Note. the above condition %looks like it 
				%	needs to be replaced with 
				%	$ w_{0}/r, w_{0}/s%\frac{w_{0}}{r}, \frac{w_{0}}{s}
				%	\in L^{4,1}(\bbR^4)$)} 
			Then, there exist $T>0$ and a unique solution to \eqref{eq_Eulereq} satisfying $w \in L^\infty(0,T%[)
			;(L^{4,1}\cap L^\infty)(\bbR^4))$ and 
			{$$(1+
				r+s)\frac{w}{rs}%w/(rs) 
				\in L^\infty(0,T%[0,T)
				;L^{4,1}(\bbR^4)).$$ %(Note. correspondingly, this needs to change to $ w/r $, $ w/s $ as well.)
			}
		\end{thm} 
		%\Red{ holds.}
		%			\begin{note}
			%					\Red{We removed the following statement at the end of Theorem \ref{thm_localreg}: \Green{The unique solution can cease to exist at some $T^*>0$ if and only if \begin{equation*}
						%								\begin{split}
							%									\limsup_{T \to T^*} \int_0^{T} \nrm{w(t)}_{L^{\infty}} dt = +\infty.
							%								\end{split}
						%						\end{equation*}}
				%Prof. KC suggests removing %stating that the $ L^{4,1} $-norms of $ w/r $ and $ w/s $ can be controlled by  instead of writing 
				%the red part in Theorem \ref{thm_localreg}. Instead, 
				%						Instead, Prof. KC suggests that we add controls of $ \nrm{w/r}_{L^{4,1}(\bbR^{4})} $ and $ \nrm{w/s}_{L^{4,1}(\bbR^{4})} $ %$ L^{4,1} $-norms of $ w/r $ and $ w/s $ 
				%						at the end of the proof of Theorem \ref{thm_globreg}. Indeed, since we know that $ \nrm{u(t)}_{L^{\ift}(\bbR^{4})} $ has an exponential estimate from Proposition \ref{prop_urLiftest} and that $ \nrm{w/(rs)}_{L^{4,1}(\bbR^{4})} $ is conserved in time, using equations \eqref{eq_woverr} and \eqref{eq_wovers} gives exponential bounds of the above two terms. The bound of the term $ \nrm{w}_{L^{4,1}(\bbR^{4})} $ follows from the equation \eqref{eq_wL4,1} and estimates of the previous two terms. -DL.}
			%				\end{note}
		In the statement above, $L^{p,q}$ refers to Lorentz spaces which are defined in the following section. It is important that the \textit{uniqueness} statement is not merely among the solutions satisfying bi-rotational symmetry but is for all possible solutions of the original Euler equations \eqref{eq_Eulereq}. 
		%			\begin{note}
			%				\Red{The green condition and result were originally $ w_{0}/(rs)\in L^{4,1}(\bbR^4) $ and $ w/(rs)\in L^\infty(0,T;L^{4,1}(\bbR^4)) $. We replaced these with \Green{green terms.} %I will explain this in the next meeting. 
				%					-DL.}
			%\Red{We are considering if we could replace the red initial condition in the above theorem with $$ (1+r+s)\frac{w_{0}}{rs}\in L^{4,1}(\bbR^{4}). $$ If so, then the proof of Theorem \ref{thm_localreg} needs to change as well.}
			%			\end{note}
		\begin{rmk}\label{rmk_blowupcrit}
			The local-in-time solution $ w $ from Theorem \ref{thm_localreg} loses its regularity at finite time $ T^{\ast}>0 $ if and only if we have
			\begin{equation}\label{eq_blowupcrit}
				\int_{0}^{T_{\ast}}\bigg[\bigg\|\frac{w(t)}{r}\bigg\|_{L^{4,1}(\bbR^{4})}+\bigg\|\frac{w(t)}{s}\bigg\|_{L^{4,1}(\bbR^{4})}\bigg]dt=+\ift.%,\quad\int_{0}^{T_{\ast}}\bigg\|\frac{w(t)}{s}\bigg\|_{L^{4,1}(\bbR^{4})}dt=+\ift.
			\end{equation}
		\end{rmk}
		Furthermore, under additional assumptions on the initial data, one can guarantee global in time well-posedness. 
		\begin{thm}\label{thm_globreg}
			Let $w_{0}$ %satisfying conditions 
			from Theorem \ref{thm_localreg} additionally satisfy %$(rs)^{-1}w_{0}\in L^{\ift}(\bbR^{4})$ and 
			$$ (1+r+s)%(1+r)(1+s)
			\frac{w_{0}}{rs}\in L^{\ift}%(L^{1}\cap L^{\ift})
			(\bbR^{4}) \qquad \mbox{and} \qquad  {(1+r^{2}+s^{2})\frac{w_{0}}{rs}\in L^{1}(\bbR^{4}).}%sw_{0}/r, rw_{0}/s, w_{0}/(rs)%
			%\frac{sw_{0}}{r},\frac{rw_{0}}{s},\frac{w_{0}}{rs}%(rs^{-1}+r^{-1}s)w_{0}
			%.
			%\bigg(1+\frac{1}{r}\bigg)\bigg(1+\frac{1}{s}\bigg)%(1+1/r)(1+1/s)
			%w_{0}
			$$ 
			Then the corresponding solution $w$ of \eqref{eq_Vorticityform} is global in time and satisfies the estimate
			\begin{align}
				\nrm{w(t)}_{L^{\ift}(\bbR^{4})} \le C \exp( Ct ) ,\label{eq_wtgrowth}\\
				\bigg\|\frac{w(t)}{r}\bigg\|_{L^{4,1}(\bbR^{4})}\leq C\exp(Ct), \label{eq_woverrexp}\\
				\bigg\|\frac{w(t)}{s}\bigg\|_{L^{4,1}(\bbR^{4})}\leq C\exp(Ct), \label{eq_woversexp}
			\end{align}
			for some constant $C>0$ depending  {only} on 
			{$$ \nrm{w_{0}}_{L^{\ift}(\bbR^{4})},\quad 
				{\bigg\|(1+r+s)\frac{w_{0}}{rs}\bigg\|_{L^{\ift}(\bbR^{4})},\quad%\bigg\|\frac{w_{0}}{r}\bigg\|_{L^{\ift}(\bbR^{4})},\quad \bigg\|\frac{w_{0}}{s}\bigg\|_{L^{\ift}(\bbR^{4})},\quad \bigg\|\frac{w_{0}}{rs}\bigg\|_{L^{\ift}(\bbR^{4})},\quad 
					\text{and}\quad \bigg\|(1+r^{2}+s^{2})\frac{w_{0}}{rs}\bigg\|_{L^{1}(\bbR^{4})}.}%\bigg\|\frac{sw_{0}}{r}\bigg\|_{L^{1}(\bbR^{4})},\quad \bigg\|\frac{rw_{0}}{s}\bigg\|_{L^{1}(\bbR^{4})},\quad\text{and}\quad \bigg\|\frac{w_{0}}{rs}\bigg\|_{L^{1}(\bbR^{4})}. 
				$$}%$w_{0}$. 
			\begin{comment}
				on 
				\Green{$$ \nrm{w_{0}}_{L^{\ift}(\bbR^{4})},\quad \bigg\|\frac{w_{0}}{r}\bigg\|_{L^{\ift}(\bbR^{4})},\quad \bigg\|\frac{w_{0}}{s}\bigg\|_{L^{\ift}(\bbR^{4})},\quad \bigg\|\frac{w_{0}}{rs}\bigg\|_{L^{\ift}(\bbR^{4})},\quad \bigg\|\frac{sw_{0}}{r}\bigg\|_{L^{1}(\bbR^{4})},\quad \bigg\|\frac{rw_{0}}{s}\bigg\|_{L^{1}(\bbR^{4})},\quad\text{and}\quad \bigg\|\frac{w_{0}}{rs}\bigg\|_{L^{1}(\bbR^{4})}. $$}%$w_{0}$. 
				
			\end{comment} 
		\end{thm}	
		
		\begin{rmk}[Global well-posedness for smooth and decaying data]\label{rmk_globaldata}
			If $ u_{0}\in H^{s}(\bbR^{4}) $ for some $ s>5 $ satisfies \eqref{eq_bi-rotational} and \eqref{eq_noswirl}, then the condition  $  {(1+r+s)}%(1+r)(1+s)
			w_{0}/(rs)%\frac{w_{0}}{rs}
			\in L^{\ift}(\bbR^{4}) $ 
			from Theorem \ref{thm_globreg} is automatically satisfied. We will prove this in Proposition \ref{prop_Liftwoverrs} from the Appendix \ref{seca:vorticity}. 
			%\Green{Indeed, from the Sobolev embedding theorem, $ u_{0} $ is in $ C^{3}(\bbR^{4}) $, which implies $ w_{0}\in C^{2}(\bbR^{4}) $. Then due to sufficiently high regularity of $ u_{0} $ and its bi-rotational symmetry, $ w_{0} $ must go to zero in a linear speed as $ r $ or $ s $ or $ rs $ approach to zero. From this, we get $ w_{0}/r, w_{0}/s, w_{0}/(rs)\in L^{\ift}(\bbR^{4}) $. }%\Red{($ << $ Needs proof.)} 
			The other condition  
			$ (1+r^{2}+s^{2})w_{0}/(rs) 
			\in L^{1}(\bbR^{4})$ is satisfied when $\nb u_{0}$ has some decay at infinity. Therefore, we conclude the following: sufficiently smooth and decaying initial data of \eqref{eq_Eulereq} have globally unique smooth solutions under bi-rotational symmetry without swirl.
		\end{rmk}
		
		\subsection{Key ideas}\label{subsec_keyideas}
		
		\textbf{Support growth of a compactly supported vorticity.} For the moment, let us consider an easier setting where the initial vorticity $ w_{0} $ is compactly supported away from the axes $ \{r=0\} $ and $ \{s=0\} $. Then using the conservation of $ L^{p} $-norm ($ p\in[1,\ift] $) of $ w/(rs) $ in time, we can estimate the $ L^{\ift} $-norm of the corresponding local-in-time solution $ w $ from Theorem \ref{thm_localreg} as
		$$ \nrm{w(t)}_{L^{\ift}(\bbR^{4})}\leq R(t)S(t)\bigg\|\frac{w_{0}}{rs}\bigg\|_{L^{\ift}(\bbR^{4})}\lesssim[R(t)+S(t)]^{2}\bigg\|\frac{w_{0}}{rs}\bigg\|_{L^{\ift}(\bbR^{4})},\quad t\in[0,T_{\text{max}}), $$
		where $ R(t) $ and $ S(t) $ are given as%(resp. $ S(t) $) is the size of support of $ w(t) $ in the $ r $ (resp. $ s $)-direction:
		$$ R(t)=R(0)+\int_{0}^{t}\nrm{u^{r}(\tau)}_{L^{\ift}(\bbR^{4})}d\tau,\quad S(t)=S(0)+\int_{0}^{t}\nrm{u^{s}(\tau)}_{L^{\ift}(\bbR^{4})}d\tau. $$
		These represent the highest distance where a point contained in the support of $ w_{0} $ can move along the flow in $ r $ and $ s $-direction, respectively. 
		This narrows our goal down to obtaining estimates of $ \nrm{u^{r}(t)}_{L^{\ift}(\bbR^{4})} $ and $ \nrm{u^{s}(t)}_{L^{\ift}(\bbR^{4})} $. This can be done by the following time-independent estimate (Proposition \ref{prop_urLiftest}):
		\begin{equation}\label{eq_timeindepest}
			\nrm{u^{r}}_{L^{\ift}(\bbR^{4})}, \, \nrm{u^{s}}_{L^{\ift}(\bbR^{4})}\lesssim\bigg\|(r^{2}+s^{2})\frac{w}{rs}\bigg\|_{L^{1}(\bbR^{4})}^{1/2}
			%\bigg[%\nrm{w}_{L^{1}(\bbR^{4})}+
			%\bigg\|\frac{rw}{s}\bigg\|_{L^{1}(\bbR^{4})}+\bigg\|\frac{sw}{r}\bigg\|_{L^{1}(\bbR^{4})}\bigg]^{1/2}
			\cdot\bigg\|\frac{w}{rs}\bigg\|_{L^{\ift}(\bbR^{4})}^{1/2}.
		\end{equation}
		Such an estimate is obtained by analyzing Biot--Savart kernels of $ u^{r} $ and $ u^{s} $, which have symmetry relation. %as described in Remark . 
		In addition, to prove the above estimate, due to the scale-invariance of the Biot--Savart formula with respect to the scaling $ u\mapsto u(\lmb\cdot) $ and $ w\mapsto \lmb w(\lmb\cdot) $ for any $ \lmb>0 $, it is sufficient enough to show
		$$ \sup_{\substack{r,s\geq0, \\ r^{2}+s^{2}=1}}|u^{r}(r,s)|, \sup_{\substack{r,s\geq0, \\ r^{2}+s^{2}=1}}|u^{s}(r,s)|\lesssim\bigg\|(r^{2}+s^{2})\frac{w}{rs}\bigg\|_{L^{1}(\bbR^{4})}^{1/2}\cdot\bigg\|\frac{w}{rs}\bigg\|_{L^{\ift}(\bbR^{4})}^{1/2}. $$
		This idea using scaling was originally from Feng--Sverak \cite{FeSv} in the 3D axisymmetric Euler without swirl. Once the estimate \eqref{eq_timeindepest} is shown, then we can get
		\begin{equation*}
			\begin{split}
				\frac{d}{dt}[R(t)+S(t)]&=\nrm{u^{r}(t)}_{L^{\ift}(\bbR^{4})}+\nrm{u^{s}(t)}_{L^{\ift}(\bbR^{4})}\lesssim_{w_{0}}[R(t)^{2}+S(t)^{2}]^{1/2}\leq R(t)+S(t).
			\end{split}
		\end{equation*}
		This gives the term $ R(t)+S(t) $ %, whose square is greater than or equal to $ R(t)S(t) $, 
		an exponential upper bound with a constant term added.
		
		Here, let us emphasize an additional difficulty arising in the proof of \eqref{eq_timeindepest}, compared to the axisymmetric case. In that situation, the velocity is a function of $r,z$, and using scaling \textit{and} translational invariance, it suffices to prove the bound \eqref{eq_timeindepest} only in the special case $(r,z) = (1,0)$. In the current setup of bi-rotational symmetry, there are serious technical difficulties arising not only from $(r,s) = (1,0)$ or $(0,1)$ but also from the whole circle $r^2+s^2=1$. 
		%we need to prove the estimate on the whole circle $r^2+s^2=1$, \Red{but also} when .  
		This is the reason why we need the control of $(r^2+s^2)w/(rs) \in L^{1}$ (and not just $rsw \in L^{1}$). 
	
		Furthermore, in the axisymmetric case, the Biot--Savart kernel can be simplified into a composition between a single-variable function and some scalar function with other simple terms multiplied. This idea was originally from Feng--Sverak \cite{FeSv} and it makes the estimate of the Biot--Savart kernel much easier as the process narrows down to estimating this single-variable function. However, in our bi-rotational symmetry setting, we cannot simplify the Biot--Savart kernel via some one variable function, so we had to directly estimate the Biot--Savart kernel to achieve our goal.
		
		\medskip
		
		\noindent \textbf{Removing the compact support assumption.} The case where the vorticity $ w $ is not necessarily compactly supported can be handled in a similar way as the above easier case. The support size $ R(t)+S(t) $ is replaced by the following ``length'' function $ L $:
		$$ L(t)=L(0)+\int_{0}^{t}[\nrm{u^{r}(\tau)}_{L^{\ift}(\bbR^{4})}+\nrm{u^{s}(\tau)}_{L^{\ift}(\bbR^{4})}]d\tau. $$
		Then instead of compactly supportedness, if we assume decay assumptions near rotational axes and at infinity, then we can obtain estimates of $ \nrm{w}_{L^{\ift}(\bbR^{4})} $ and the first term on the right-hand side of \eqref{eq_timeindepest}:
		\begin{equation*}
			\nrm{w(t)}_{L^{\ift}(\bbR^{4})}\lesssim_{w_{0}}L(t)^{2}, \qquad 					\bigg\|(r^{2}+s^{2})\frac{w}{rs}\bigg\|_{L^{1}(\bbR^{4})}\lesssim_{w_{0}}L(t)^{2}. 
		\end{equation*} 
		Using these and \eqref{eq_timeindepest}, we can get an exponential bound of $ L(t) $ by the following inequality,
		$$ \frac{d}{dt}L(t)=\nrm{u^{r}(t)}_{L^{\ift}(\bbR^{4})}+\nrm{u^{s}(t)}_{L^{\ift}(\bbR^{4})}\lesssim_{w_{0}}L(t), $$
		which immediately implies that $ \nrm{w}_{L^{\ift}(\bbR^{4})} $ is exponentially bounded in time.
		%Then once again using the time-independent estimate \eqref{eq_timeindepest}  
		
		\subsection{{Literature review}}\label{subsec_lit}
		
		To put our results in context, we review relevant literature, starting from the axisymmetric Euler equations.

		\subsubsection{Axisymmetric solutions in three dimensions}
		
		In this section, we briefly review the theory of axisymmetric solutions to the 3D Euler equations, which naturally motivates the results in the current paper. We fix the domain to be $\bbR^3$ and use cylindrical coordinates $(r,\tht,z)$.  %Axisymmetry means that $u(Ox) = Ou(x)$ for any rotation $O$ of $\bbR^3$ which fix the $z$-axis: we have \begin{equation*}
			%	\begin{split}
				%		u = u^r(r,z) e^r + u^\tht(r,z) e^\tht +  u^z(r,z) e^z. 
				%	\end{split}
			%\end{equation*} 
			Under axisymmery, 3D Euler equations reduce to the system involving the axial components of the velocity and its curl, denoted by $u^\tht$ and $\omg^\tht =\rd_zu^r - \rd_ru^z $ (\cite{MB}): \begin{equation}  \label{eq:3D-axisym-Euler}
				\left\{
				\begin{aligned} 
					&(\rd_t  + u^r\rd_r + u^z \rd_z ) (ru^\tht )=0 , \\
					&(\rd_t  + u^r\rd_r + u^z \rd_z ) (\frac{\omg^\tht}{r} )= \frac{1}{r^4} \rd_z [(ru^\tht)^2],
				\end{aligned}
				\right.
			\end{equation} supplemented with the divergence-free condition $\rd_ru^r + \frac{u^r}{r} + \rd_z u^z = 0.$

			\medskip 
			
			\noindent \textit{Global well-posedness in the no-swirl case}. %The axial component of the velocity is often referred to as swirl, and it has been well-known that in the no-swirl case, the axisymmetric Euler equations admit global solutions. 
			While the problem of singularity formation for smooth solutions to \eqref{eq:3D-axisym-Euler} is still open, global regularity in the no-swirl case $u^\tht=0$ is well-known. In this case, the system is simply \begin{equation}\label{eq:3D-axisym-Euler-no-swirl}
				\begin{split}
					(\rd_t  + u^r\rd_r + u^z \rd_z ) (\frac{\omg^\tht}{r} )= 0, 
				\end{split}
			\end{equation} so that any $L^{p,q}$-norm (with respect to the Lebesgue measure on $\bbR^3$) of $\frac{\omg^\tht}{r} $ is conserved in time. This equation is analogous to the 2D vorticity equation, which is simply the transport equation for the vorticity. In the 2D case, Yudovich theory \cite{Y1} gives global well-posedness for vorticities belonging to $L^\infty(0,\ift%[0,\infty)
		; L^\infty \cap L^1(\bbR^2))$. Ideally, one would like to prove well-posedness for \eqref{eq:3D-axisym-Euler-no-swirl} with merely bounded and decaying $\omg^\tht$ as well, but there are subtle differences due to the factor $r^{-1}$.
			
			Regarding \eqref{eq:3D-axisym-Euler-no-swirl}, existence of global solutions was first proved in Ukhovsksii--Yudovich \cite{UY} under the assumptions $\omg^\tht_0, \frac{\omg^\tht_0}{r} \in L^2\cap L^\infty$. Later, global well-posedness in the case $\omg_0^\tht \in H^s$ with $s>\frac{3}{2}$ was obtained in \cite{Majda,Raymond,SY}. Then, Danchin proved global regularity under the assumption $\omg^\tht_0 \in L^\infty \cap L^{3,1}$, $\frac{\omg^\tht_0}{r} \in L^{3,1}$ in \cite{Danaxi}, and this seems to be optimal. 
			
			To briefly explain the difference with the 2D Euler equation, let us return to the equation for $\omg^\tht$ and observe the estimate \begin{equation*}
				\begin{split}
					\frac{d}{dt}\nrm{\omg^\tht}_{L^\infty} \le \nrm{r^{-1}u^r}_{L^\infty} \nrm{\omg^\tht}_{L^\infty}.
				\end{split}
			\end{equation*} However, having $\omg^\tht$ in $L^\infty \cap L^1$ does not guarantee boundedness of $r^{-1}u^r$. On the other hand, estimate due to Shirota--Yanagisawa (\cite{SY}) and Danchin (\cite[Lemma 2]{Danaxi}) gives  \begin{equation*}
				\begin{split}
					\frac{|u^r(r,z)|}{r} \le C\int_{\mathbb{R}^3} \frac{1}{|x-\bar{x}|^2} \frac{|\omg^\tht(\bar{x})|}{\bar{r}} d\bar{x} \le C\nrm{r^{-1}\omg^\tht}_{L^{3,1}}= C\nrm{r^{-1}\omg^\tht_{0}}_{L^{3,1}},
				\end{split}
			\end{equation*} which then allows us to get \begin{equation*}
				\begin{split}
					\frac{d}{dt}\nrm{\omg^\tht}_{L^\infty} \le C\nrm{r^{-1}\omg^\tht_0}_{L^{3,1}} \nrm{\omg^\tht}_{L^\infty}.
				\end{split}
			\end{equation*} This guarantees that the vorticity remains uniformly bounded until any finite time, and this gives not only uniqueness but also propagates higher H\"older or Sobolev regularity of the vorticity.
			
			We note that the condition $\frac{\omg^\tht_0}{r} \in L^{3,1}(\bbR^3)$ is automatic (via Sobolev embedding) when the initial vorticity belongs to either $H^s(\bbR^3)$ with $s>\frac{3}{2}$ or $C^\alp(\bbR^3)$ with $\alp>\frac{1}{3}$. We note that global well-posedness for smooth and axisymmetric vortex patch solutions to \eqref{eq:3D-axisym-Euler-no-swirl} was proved in \cite{Hu} for initial data satisfying $|\omg_0^\tht(r,z)|\lesssim |r|^\beta$ for $\beta>\frac{1}{3}$. Moreover, global well-posedness of \eqref{eq:3D-axisym-Euler-no-swirl} in certain critical Besov spaces has been established in \cite{AHK,Cozzi,WuG}. 
			
			On the other hand, optimality of $\frac{\omg^\tht_0}{r} \in L^{3,1}(\bbR^3)$ can be seen by directly using the kernel expression for $u^r/r$: for any $\eps>0$, it is possible to find $\omg^\tht_0 \in L^1 \cap L^\infty (\bbR^3)$ such that $\frac{\omg^\tht_0}{r} \in L^{3,1+\eps}(\bbR^3)$ but the corresponding velocity satisfies $u^{r}_{0}/r \notin L^{\infty}(\bbR^3)$. Indeed, strong illposedness of \eqref{eq:3D-axisym-Euler-no-swirl} was obtained in critical Lorentz spaces in \cite{JK-axi} based on this observation. Earlier work of Bourgain--Li \cite{BL3D} obtained strong illposedness for \eqref{eq:3D-axisym-Euler-no-swirl} in the critical $H^{3/2}(\bbR^3)$ space.

			\medskip 
			
			\noindent \textit{Finite-time singularity formation for $C^\alpha$ data}. The aforementioned works may give the impression that the axisymmetric no-swirl equation is globally well-posed in any function space in which the equation is locally well-posed. Essentially\footnote{Strictly speaking, it is not clear if the BKM criterion is valid for all spaces in which the Euler equations are well-posed (see \cite{Danaxi}).}, Danchin's result shows that this is correct once $r^{-1}\omg_0^\tht \in L^{3,1}$. However, recall that for $\omg_0^\tht \in C^\alpha$, $\omg^\tht_0 \in L^{3,1}$ is guaranteed only when $\alpha>\frac{1}{3}$. Remarkably, recent work of Elgindi \cite{Elgindi-3D} (see also \cite{EGM}) proved that finite-time blow up occurs for the no-swirl equation with initial data $\omg_0^\tht \in C^\alp$ with $0<\alp$ small. Although singularity formation for the axisymmetric no-swirl system is limited to vorticities less regular than $C^{\frac{1}{3}}$, the situation could be different in other geometries. It would be interesting to investigate Elgindi type of singularity formation for the bi-rotational geometry as well. However, in the current work, the initial data considered are regular enough near the axis (e.g. $w \in L^{4,1}(\bbR^4)$) which rules out this kind of singularity formation. Inspired by this work, Chen--Hou \cite{ChenHou} showed finite-time singularity formation of some solution of 3D axisymmetric Euler equations with swirl in a bounded cylinder with $ C^{\alp} $-initial vorticity. 
			\begin{comment}
				\Red{(I will update Chen--Hou's %--Luo 
					$ C^{\alp} $-type 
					singularity result \cite{ChenHou} soon. %to be added here. 
					-DL.)}
			\end{comment}

			%This motivates the study of the ``no-swirl'' system in other geometries, the bi-rotational symmetry just being an example. 
			
			\subsubsection{Axisymmetric Euler equations in higher dimensions}
			
			%\Blue{	Need to also cite: \url{https://arxiv.org/abs/2212.11912}, \url{https://arxiv.org/abs/2212.11924} }

			As we mentioned earlier, axisymmetric solutions can be considered for the Euler equations defined in higher dimensions as well. Interestingly, it turns out that as dimension becomes larger, there seems to be a higher chance of singularity formation, even for smooth and decaying initial data.
			
			Drivas--Elgindi \cite{DE} considered pressureless Euler equations in each dimension. They constructed a sequence of solutions for each dimension that is initially smooth and blows up in finite time. Then, they sent the dimension to infinity to show that there is a finite-time blowup in the ``infinite dimension'', in some sense. This lead them to propose an open question about whether there would be a finite-time blowup of initially smooth, axisymmetric, and swirl-free solution of the Euler equations in some high dimension. Miller \cite{Miller} also raised the same open question. In addition, Hou--Zhang \cite{HouZhang22P1,HouZhang22P2} investigated potential finite-time singularity of the axisymmetric Euler equations without swirl where the initial vorticity is in $ C^{\alp} $ for a large range of $ \alp $ in both 3D and higher dimensions.
			%Miller \cite{Miller} constructed 'infinite dimensional' axisymmetric Euler equation without swirl, which has the form of 1D Burgers equation, by suitably modifying the Euler equations and taking the dimension to infinity. Then he proved that there is a solution of the equation with smooth initial data that blows up in finite time, which lead to suggest the possibility of finite-time blowup in some high dimension. Drivas--Elgindi \cite{DE} also raised the same open question. %regarding the possibility of finite-time blowup of some initially-smooth solution of axisymmetric Euler equations in high dimensions.
			%\medskip	
			
			%\noindent 
			On the other hand, there are several works regarding global existence of locally well-posed smooth 
			solutions (\cite{Miller}) with additional decay assumptions at the origin and at infinity. 
			In our previous work \cite{CJLglobal22}, it was shown that such decay assumptions are sufficient for global regularity to hold in $ \bbR^{4} $. However, the decay assumption near the symmetry axis used in \cite{CJLglobal22} is not automatically satisfied for smooth data in $\bbR^{4}$, in stark contrast to our global wellposedness result. In a more recent work by the third author \cite{Limglobal23}, it was shown that for any $ d\geq3 $, a \textit{single-signed} vorticity with compact support and fast decay at the origin is globally regular. Indeed, an explicit estimate of the size of vorticity confinement can be obtained. 
			
			\subsubsection{The lake equation and the vortex membrane equation}
			
			There is another equation that is closely related to the bi-rotational Euler equations, namely,
			%Works of Khesin--Yang \cite{KhYa} and Yang \cite{Yang} presented that for integers $m,l\in\bbN$, the Euler equations in $\bbR^{m+l+2}$ with sphere product $\bbS^{m}\times\bbS^{l}$-symmetry (that is, when the velocity $u$ and the pressure $p$ depend only on distances $|x|, |y|$ with $x\in\bbR^{m+1}, y\in\bbR^{l+1}$ to the origin) can be described as 
			the lake equation in $ \Omg\subset\bbR^{2} $;
			\begin{equation}\label{eq_lakeeq}
				\begin{split}
					\rd_{t} {\bigg(\frac{\omg}{b}\bigg)}+U\cdot\nb {\bigg(\frac{\omg}{b}\bigg)}&=0\quad\text{in}\quad\Omg,\\
					\nb\cdot(bU)&=0\quad\text{in}\quad\Omg,\\
					U\cdot n&=0\quad\text{in}\quad\rd\Omg,\\
					\omg&=-\rd_{2}U^{1}+\rd_{1}U^{2},
				\end{split}
			\end{equation}
			%(see \cite{LOT96, DVS20, HLM22}) 
			with a nonnegative depth function $ b : \Omg\to[0,\ift) $. %$b(|x|,|y|)=|x|^{m}|y|^{l}$. 
			Note that when the domain $ \Omg $ is the half plane $ \lbrace (r,z) : r\geq0, z\in\bbR\rbrace $ with the depth function chosen as $ b(r,z)=r $, then this corresponds to the 3D axisymmetric Euler equations without swirl. Additionally, if we take $ \Omg $ as the first quadrant $ \lbrace(r,s) : r,s\geq0\rbrace $ and $ b $ as $ b(r,s)=rs $, then this is exactly the bi-rotational Euler equations that we are considering in this paper. There are several results about global existence of Yudovich-type solutions in bounded domains (\cite{LOT96,LNP14}).

			% \Red{(I will update vortex membrane equations soon. %to be added here. 
				%-DL.)}
			
			%The case $m=l=1$ corresponds to the bi-rotational symmetry. 
			%the bi-rotational Euler equation \eqref{eq_Vorticityform} without swirl can be written as the lake equation (see \cite{LOT96, DVS20, HLM22}) with the depth function $b(r,s)=rs, \ r,s\geq0$. 
			Also, works of Khesin--Yang \cite{KhYa} and Yang \cite{Yang} presented that for integers $m,l\in\bbN$, the Euler equations in $\bbR^{m+l+2}$ with sphere product $\bbS^{m}\times\bbS^{l}$-symmetry (that is, when the velocity $u$ and the pressure $p$ depend only on distances $|x|, |y|$ with $x\in\bbR^{m+1}, y\in\bbR^{l+1}$ to the origin) can be written as the lake equation \eqref{eq_lakeeq} in the first quadrant $ \lbrace(|x|,|y|)\rbrace $ with $b(|x|,|y|)=|x|^{m}|y|^{l}$. We can consider the Euler equations with such symmetry as the bi-rotational Euler equations in higher dimensions $ \bbR^{ {m+l+2}} $, with the equation
			\begin{equation}\label{eq_birothighdim}
				\rd_{t}\bigg(\frac{w}{r^{m}s^{l}}\bigg)+(u^{r}\rd_{r}+u^{s}\rd_{s})\bigg(\frac{w}{r^{m}s^{l}}\bigg)=0,\quad w=-\rd_{s}u^{r}+\rd_{r}u^{s}.
			\end{equation}
			In addition, their works show that the motion of a point vortex of a point $(a,b)$ in the first quadrant for the equation \eqref{eq_birothighdim} can be regarded as \textit{a sphere product vortex membrane} under \textit{skew-mean-curvature flow}, and %. Moreover, 
			they provided explicit point-vortex type solutions which exist for all time or shows finite-time blowup with dependence on $m, l$. Here, a \textit{sphere product vortex membrane} is a singular vorticity supported on the sphere product $\bbS^{m}(a)\times\bbS^{l}(b)$ with radii $a,b\geq0$, and the \textit{skew-mean-curvature flow} is described by the following equation:
			$$ \rd_{t}q=-J(H(q)),\quad q\in\Sgm, $$
			where $\Sgm^{n}\subset\bbR^{n+2}$ is a codimension 2 compact oriented submanifold(membrane) in $\bbR^{n+2}$, $H(q)$ is the mean curvature vector at the point $q\in\Sgm$, and $J$ is the operator of positive $\pi/2$ rotation in the 2-dimensional normal space $N_{q}\Sgm$ to $\Sgm$ at $q$. This is a natural generalization from the case $ n=1 $, which is the binormal equation:
			\begin{equation}\label{eq_binormal}
				\rd_{t}\gmm=\kpp\textbf{b},
			\end{equation}
			where $ \gmm $ is a curve in $ \bbR^{3} $, $ \kpp $ is the curvature of $ \gmm $ at some point, and $ \textbf{b} $ is the binormal vector at that point. Indeed, when $ n=1 $, the mean curvature vector of $ \gmm $ at a point is given as $ H=\kpp\textbf{n} $, with $ \textbf{n} $ being the outward normal vector at that point, so we have \eqref{eq_binormal}: $ \rd_{t}\gmm=-J(\kpp\textbf{n})=\kpp\textbf{t}\times\textbf{n}=\kpp\textbf{b} $ ($ \textbf{t} $ is the tangent vector). However, it is not clear whether this dynamics of point vortices can be well-approximated by that of smooth solutions. 
			%Additionally, those works show that one can regard a sphere product vortex membrane, which is a singular vorticity supported on the sphere product $\bbS^{m}(a)\times\bbS^{l}(b)$ with radii $a,b\geq0$, under binormal %skew-mean-curvature 
			%flow %(a natural generalization of the binormal equation) 
			%as the motion of a point vortex of the point $(a,b)$ in the lake equation.

			\subsubsection{Comparison with 2D and 3D Euler equations}

			\begin{figure}
				\centering
				\includegraphics[scale=0.8]{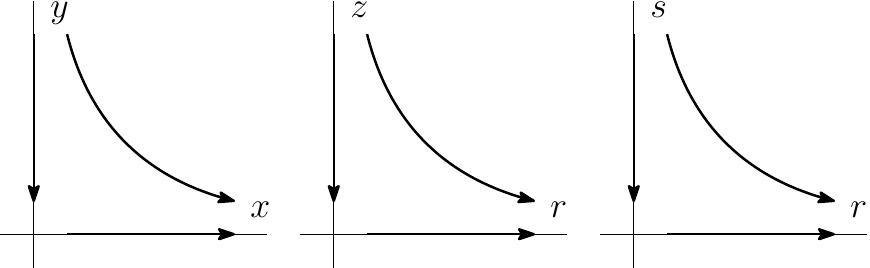} 
				\caption{Overall flow pattern for signed vorticity: $\bbR^{2}$ under odd-odd symmetry (left), $\bbR^{3}$ under axisymmetry and odd symmetry in $z$ (center), $\bbR^{4}$ under bi-rotational symmetry (right).} \label{figure:flow}
			\end{figure}

			The bi-rotational Euler system can be considered in the first quadrant $ \lbrace(r,s) : r,s\geq0\rbrace $ with $ r=0 $ and $ s=0 $ as rotational axes. This is similar with 
			%This is why we can regard the bi-rotational Euler as an extension of 
			2D Euler with odd-odd symmetry and 3D axisymmetric, swirl-free Euler with odd symmetry in $ z $, because under such symmetry assumptions, the domain again reduces to the first quadrant. Furthermore, under the additional assumption that the vorticity is \textit{signed}, then in all three cases, it can be shown that in an averaged sense, the overall fluid flow is directed southeast; see Figure \ref{figure:flow}. 	 In 2D Euler with odd-odd symmetry,  Iftimie--Sideris--Gamblin \cite{ISG99} proved that the horizontal center of mass %size of the support of the vorticity 
			grows linearly in time, using single-signed vorticity. Similarly, in 3D axisymmetric Euler without swirl and with $ z $-odd symmetry, the first and the second author \cite{CJ-axi} showed sublinear growth of impulse of a single-signed vorticity in time. This was enhanced by Gustafson--Miller--Tsai \cite{GMT2023}. 
			One of the main difference from axisymmetric Euler in 4D is that there seems to be no time-conserved quantity corresponding to angular momentum.% This could be an obstacle if 
			
			We would like to point out an additional difference of bi-rotational symmetry to axisymmetry: the variables $r$ and $s$ are symmetric, as in the case of 2D Euler. Therefore, one may consider solutions which satisfy an additional odd symmetry assumption across the diagonal line $\{ r = s \}$. If one furthermore considers signed vorticity (namely, $w \ge 0$ in the region $\{ s > r\}$), then the corresponding flow pattern looks like the one depicted in Figure \ref{figure:flow2}.

			\begin{figure}
				\centering
				\includegraphics[scale=0.8]{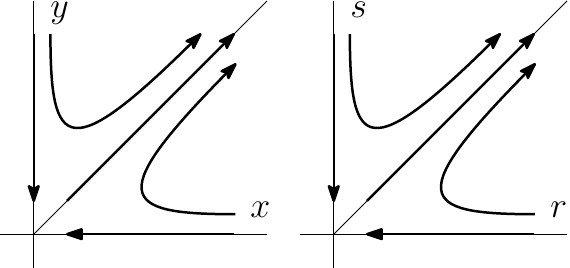} 
				\caption{Overall flow pattern for signed vorticity under an additional odd symmetry} \label{figure:flow2}
			\end{figure}

			\subsection{Outline of the paper} \label{subsec:outline}
			
			The rest of the paper is organized as follows. In section \ref{sec:prelim}, we provide basic definition and properties of the Lorentz space. Then, we introduce the vorticity tensor in the bi-rotational geometry and derive the Biot--Savart law of the bi-rotational Euler flow. In section \ref{sec_LWP}, we conduct estimates of the Biot--Savart kernel that was obtained in section \ref{sec:prelim} to show Theorem \ref{thm_localreg}, the local regularity of $ w $. In section \ref{sec_GWP}, the last section of this paper, we prove Theorem \ref{thm_globreg}, the global regularity of $ w $, by doing finer estimate of the Biot--Savart kernel.
			
			\section{Preliminaries}\label{sec:prelim}

			In this section, after giving a brief review of Lorentz spaces, we derive the Biot--Savart law for the bi-rotational Euler equations. 
			
			\subsection{Lorentz spaces}\label{subsec:Lorentz}
			
			We briefly review the definition of Lorentz spaces %(\cite{Lorentz}) 
			and a few basic properties; a classical reference is \cite{GrafakosFourier}. 
			\begin{defn}\label{def:Lorentz}
				Given a function $f:\bbR^n\rightarrow\bbR$, we define $f^*:\bbR_+\rightarrow\bbR_+$ by $f^*(t) = \inf\{ t>0: |\{ x \in\bbR^n : |f(x)|>\lmb \}|\le t \}.$ We then say $f \in L^{p,q}$ for $1\le p,q \le \infty $ if the quantity \begin{equation*}
					\begin{split}
						\nrm{f}_{L^{p,q}} = \begin{cases}
							\left( \int_0^\infty (t^{\frac{1}{p}}f^*(t))^q \frac{dt}{t} \right)^{\frac{1}{q}}, \quad & q<\infty,\\
							\sup_{t>0} \, t^{\frac{1}{p}}f^*(t), \quad & q = \infty. 
						\end{cases}
					\end{split}
				\end{equation*} is finite. We write $L^{p,q}$ as the set of all measurable functions with $\nrm{f}_{L^{p,q}}<\infty$. 
			\end{defn}
			Although the spaces can be defined even when either $p<1$ or $q<1$, we shall not be concerned with such cases. In the following we shall list a few properties which will be useful. The proofs can be found in \cite[Section 1.4]{GrafakosFourier} and the references therein. 
			\begin{prop}
				The spaces $L^{p,q}$ satisfy the following properties: \begin{itemize}
					\item The space $L^{p,p}$ coincides with the usual $L^p$ space. 
					\item The functional $\nrm{\cdot}_{L^{p,q}}$ is a quasi-norm; that is, $\nrm{f+g}_{L^{p,q}}  \le C_{p,q} \left( \nrm{f}_{L^{p,q}}+\nrm{g}_{L^{p,q}} \right).$
					\item With $1/p+1/p^* = 1$ and $1/q + 1/q^* = 1$, we have $(L^{p,q})^* = L^{p^*,q^*}$ for any $p>1$ and $1\le q<+\infty$. 
					\item We have the H\"older's inequality $\nrm{fg}_{L^{r,s}} \le C_{p,q,s_1,s_2}\nrm{f}_{L^{p,s_1}} \nrm{g}_{L^{q,s_2}}$ for any $1\le p,q,r,s_1,s_2\le+\infty$ satisfying $\frac{1}{p} + \frac{1}{q} = \frac{1}{r},\,  \frac{1}{s_1} + \frac{1}{s_2} = \frac{1}{s}. $
				\end{itemize}
			\end{prop}

			\subsection{Vorticity formulation of the Euler equations}
			
			Introducing the vorticity tensor $\omg = (\omg^{i,j})_{1\le i,j \le n}$ where $\omega^{i,j} = \rd_{x_j}u^i-\rd_{x_i}u^j,$ \eqref{eq_Eulereq} becomes \begin{equation}\label{eq:Euler-vort}
				\begin{split}
					\rd_t\omega^{i,j} + u\cdot\nabla \omega^{i,j} = - \sum_{k=1}^n \left( \omega^{k,j}\rd_{x_k}u^i - \omega^{k,i}\rd_{x_k}u^j \right). 
				\end{split}
			\end{equation}  The components of $\omg$ are subject to the following consistency relations: \begin{equation}\label{eq:consist}
				\begin{split}
					\rd_{x_k}\omg^{i,j} + \rd_{x_i}\omg^{j,k} + \rd_{x_j}\omg^{k,i} = 0. 
				\end{split}
			\end{equation} To close the system, we introduce $\Psi^{i,j} = \Delta^{-1}\omega^{i,j}$: then, one can see that $u^i = \sum_{k=1}^n \rd_{x_k} \Psi^{i,k}$  and  \begin{equation}\label{eq:4D-Riesz}
				\begin{split}
					\rd_{x_j}u^i = -\sum_{k=1}^n R_{jk}\omega^{i,k}  
				\end{split}
			\end{equation} where $R_{jk} = \rd_{x_j}\rd_{x_k}(-\lap)^{-1}$. The system \eqref{eq:Euler-vort} supplemented with \eqref{eq:consist} and \eqref{eq:4D-Riesz} is the vorticity form of the incompressible Euler equations (\cite{Chemin}). 
			
			\subsection{Vorticity tensor in the bi-rotational geometry and derivation of the Biot--Savart law}
			
			In the bi-rotational case, the vorticity tensor can be written as a linear combination of $w$ and the derivatives of $u^{\tht}$ and $u^{\phi}$, with coefficients depending on $\tht$ and $\phi$. The formulas are recorded in the Appendix \ref{seca:vorticity}. In the no-swirl case, that is, when $u^{\tht} = u^{\phi} = 0$, we have that $\omg^{1,2} = \omg^{3,4} = 0$ and \begin{equation*}
				\begin{split}
					\omg^{1,3} = -\cos\tht\cos\phi w, \quad \omg^{1,4} = -\cos\tht\sin\phi w, \quad \omg^{2,3} = -\sin\tht\cos\phi w , \quad \omg^{2,4} = -\sin\tht\sin\phi w . 
				\end{split}
			\end{equation*} We are now ready to obtain the kernel expressions for the ``radial'' velocities $u^{r}$ and $u^{s}$ in terms of $w$, which will be our Biot-Savart law. (Observe that $u^{r}$ and $u^{s}$ do not depend on $u^{\tht}$ and $u^{\phi}$; we shall therefore assume that $u^{\tht} = u^{\phi} = 0$.) Our starting point is \begin{equation*}
				\begin{split}
					u^i = \sum_{j=1}^{4} \rd_{x_j} \lap^{-1}\omg^{i,j}= \sum_{j=1}^{4} \lap^{-1} \rd_{x_j}\omg^{i,j}. 
				\end{split}
			\end{equation*} Let us recall that $\lap^{-1}$ is defined in $\bbR^4$ by convolution against the kernel \begin{equation*}
				\begin{split}
					G(x) = -\frac{1}{4\pi^2}|x|^{-2}. 
				\end{split}
			\end{equation*} Hence the kernel for $\rd_{x_j}\lap^{-1}$ is given by $K = (K_j)_{1\le j\le 4}$ with \begin{equation*}
				\begin{split}
					K_j(x) = \frac{1}{2\pi^2} \frac{x_j}{|x|^4}. 
				\end{split}
			\end{equation*} We now compute, using the formulas in the previous section, \begin{equation*}
				\begin{split}
					\rd_{x_3}\omg^{2,3} &= - \sin\tht \cos^2\phi \rd_{s}w - \sin\tht \sin^2\phi \frac{w}{s}, \qquad \rd_{x_4}\omg^{2,4} = - \sin\tht w \sin^2\phi \rd_{s}w - \sin\tht\cos^2\phi \frac{w}{s}. 
				\end{split}
			\end{equation*} In particular, \begin{equation*}
				\begin{split}
					\frac{u^2}{x_2} = -\frac{1}{x_2} \lap^{-1}\left( \sin\tht (\rd_{s}w + \frac{w}{s}) \right) =  -\frac{1}{r} \widetilde{\lap}^{-1}\left(  \rd_{s}w + \frac{w}{s} \right) = \frac{u^1}{x_1}. 
				\end{split}
			\end{equation*} Here, \begin{equation*}
				\begin{split}
					\widetilde{\lap} = \rd_{r}^{2} + \frac{\rd_{r}}{r} - \frac{1}{r^2} + \rd_{s}^{2} + \frac{\rd_{s}}{s} - \frac{1}{s^2} 
				\end{split}
			\end{equation*} Similarly, we obtain \begin{equation*}
				\begin{split}
					\frac{u^3}{x_3} = \frac{1}{x_3} \lap^{-1}\left( \cos\tht (\rd_{r}w + \frac{w}{r}) \right) =  \frac{1}{s}\widetilde{\lap}^{-1}\left(  \rd_{r}w + \frac{w}{r}  \right) = \frac{u^4}{x_4}. 
				\end{split}
			\end{equation*} Note that, under the no-swirl assumption, \begin{equation*}
				\begin{split}
					\frac{u^2}{x_2} = \frac{u^1}{x_1} = \frac{u^{r}}{r}, \quad \frac{u^3}{x_3} = \frac{u^4}{x_4} = \frac{u^{s}}{s}. 
				\end{split}
			\end{equation*} It follows that \begin{equation*}
				\begin{split}
					\frac{u^{r}}{r} + \frac{u^{s}}{s} = - \frac{1}{r}\widetilde{\lap}^{-1}\left(  \rd_{s}w + \frac{w}{s} \right) + \frac{1}{s}\widetilde{\lap}^{-1}\left(  \rd_{r}w + \frac{w}{r} \right)
				\end{split}
			\end{equation*}
			Based on the computations above, we shall derive the following expression for the kernel of $u^{r}$ and $u^{s}$. We denote $ \Pi:=\lbrace (r,s)\in\bbR^{2} : r,s\geq0\rbrace $.
			\begin{lem}\label{lem:4D-Biot-Savart}
				We have  		
				\begin{equation}\label{eq_urform}
					u^{r}(r,s)=-\iint_{\Pi}F^{r}(r,s,\br{r},\br{s})w(\br{r},\br{s})d\br{s}d\br{r},
				\end{equation}
				with
				\begin{align}\label{eq_urkernel}
					F^{r}(r,s,\br{r},\br{s})%&\Blue{=\frac{1}{2\pi^{2}}\int_{0}^{2\pi}\int_{0}^{2\pi}\frac{\br{r}\br{s}\cos\br{\tht}(\br{s}-s\cos\br{\phi})}{[(r-\br{r})^{2}+(s-\br{s})^{2}+2r\br{r}(1-\cos\br{\tht})+2s\br{s}(1-\cos\br{\phi})]^{2}}d\br{\phi} d\br{\tht}}\nonumber\\
					&=\frac{2}{\pi^{2}}\int_{0}^{\pi}\int_{0}^{\pi}\frac{\br{r}\br{s}\cos\br{\tht}(\br{s}-s\cos\br{\phi})}{[(r-\br{r})^{2}+(s-\br{s})^{2}+2r\br{r}(1-\cos\br{\tht})+2s\br{s}(1-\cos\br{\phi})%r^{2}+\br{r}^{2}-2r\br{r}\cos\tht+s^{2}+\br{s}^{2}-2s\br{s}\cos\phi
						]^{2}}d\br{\phi} d\br{\tht},
				\end{align}
				and
				\begin{equation}\label{eq_usform}
					u^{s}(r,s)=\iint_{\Pi}F^{s}(r,s,\br{r},\br{s})w(\br{r},\br{s})d\br{s}d\br{r},
				\end{equation}
				with $ F^{s}(s,r,\br{s},\br{r})=F^{r}(r,s,\br{r},\br{s}). $
			\end{lem}
			
			\begin{proof}[Proof of Lemma \ref{lem:4D-Biot-Savart}] Let us briefly comment on the notations: $x, \bar{x} \in \bbR^4$ will be expressed by $(r,\tht,s,\phi)$ and $(\bar{r},\bar{\tht},\bar{s},\bar{\phi})$ in the bi-polar coordinates, respectively. 
				Recall from above that, in the rectangular coordinate system, 
				\begin{equation*}
					\begin{split}
						(\frac{u^{r}}{r})(x) = -\frac{1}{4\pi^2 x_1} \int_{\bbR^4} \frac{\cos\bar{\tht}}{|x-\bar{x}|^2} \left( \rd_{s} w + \frac{w}{\bar{s}} \right) d\bar{x} .
					\end{split}
				\end{equation*} Rewriting in bi-polar coordinates and integrating by parts give that \begin{equation*}
					\begin{split}
						\frac{1}{ x_1} \int_{\bbR^4} \frac{\cos\bar{\tht}}{|x-\bar{x}|^2} \left( \rd_{s} w + \frac{w}{\bar{s}} \right) d\bar{x}&= \frac{1}{r\cos\tht} \int_0^\infty\int_0^\infty\int_0^{2\pi}\int_0^{2\pi} \cos\bar{\tht} \left( -\rd_{\bar{s}}\left(\frac{\bar{r}\bar{s}}{|x-\bar{x}|^2} \right) + \frac{\bar{r}}{|x-\bar{x}|^2} \right) w(\bar{r},\bar{s}) d\bar{\tht}d\bar{\phi}d\bar{r}d\bar{s}  \\
						&= \frac{1}{r\cos\tht} \int_0^\infty\int_0^\infty\int_0^{2\pi}\int_0^{2\pi} 2\bar{r}\bar{s}\cos\bar{\tht} \frac{\bar{s} - s\cos(\bar{\phi}-\phi)}{|x-\bar{x}|^4}w(\bar{r},\bar{s}) d\bar{\tht}d\bar{\phi}d\bar{r}d\bar{s}. 
					\end{split}
				\end{equation*} Note that $|x-\bar{x}|^2 = r^2+\bar{r}^2 -2r\bar{r} \cos(\bar{\tht}-\tht) +  s^2+\bar{s}^2 -2s\bar{s} \cos(\bar{\phi}-\phi)$ and a  change of variables gives that \begin{equation*}
					\begin{split}
						(\frac{u^{r}}{r})(x) & =\frac{1}{4\pi^{2}r\cos\tht}  \int_0^\infty\int_0^\infty\int_0^{2\pi}\int_0^{2\pi}%\iiiint 
						2\bar{r}\bar{s} \frac{ (\cos\bar{\tht}\cos\tht - \sin\bar{\tht}\sin\tht)(\bar{s} - s\cos\bar{\phi})}{(r^2+\bar{r}^2 -2r\bar{r} \cos\bar{\tht} +  s^2+\bar{s}^2 -2s\bar{s} \cos\bar{\phi})^2}w(\bar{r},\bar{s})d\bar{\tht}d\bar{\phi}d\bar{r}d\bar{s} \\
						& = \frac{2}{\pi^{2}r}  %\iiiint 
						\int_0^\infty\int_0^\infty\int_0^{\pi}\int_0^{\pi}\frac{ 2\bar{r}\bar{s}  \cos\bar{\tht}   (\bar{s} - s\cos\bar{\phi})}{(r^2+\bar{r}^2 -2r\bar{r} \cos\bar{\tht} +  s^2+\bar{s}^2 -2s\bar{s} \cos\bar{\phi})^2}w(\bar{r},\bar{s})d\bar{\tht}d\bar{\phi}d\bar{r}d\bar{s}.
					\end{split}
				\end{equation*} We have used that the integral term involving $\sin\tht$ vanishes after integrating in $\bar{\tht}$. After normalization, this gives the desired expression in \eqref{eq_urkernel}. % when $i = 1$. 
				The proof for \eqref{eq_usform} %the case $i=2$ 
				is parallel. 
			\end{proof}

			\section{Local regularity of $ w $}\label{sec_LWP}
			In this section, we estimate supremums of terms related to the velocity field $ u $ using the Biot--Savart kernels obtained from Lemma \ref{lem:4D-Biot-Savart} to prove Theorem \ref{thm_localreg}.
			
			\subsection{Estimates based on the Biot-Savart law}
			
			First, we present the following lemma, which says that supremums of $ u $, $ r^{-1}u^{r} $, and $ s^{-1}u^{s} $ are bounded by $ L^{4,1}- $norms of $ w $, $ r^{-1}w $, and $ s^{-1}w $, respectively.
			
			\begin{lem}\label{lem:vel-estimate1}
				We have the estimates \begin{equation}\label{eq:vel-Lorentz}
					\begin{split}
						\nrm{u^{r}}_{L^\infty}, \nrm{u^{s}}_{L^\infty} \le C \nrm{w}_{L^{4,1}}. 
					\end{split}
				\end{equation} \begin{equation}\label{eq:velratio-Lorentz}
					\begin{split}
						\nrm{r^{-1}u^{r}}_{L^\infty} \le C \nrm{r^{-1}w}_{L^{4,1}},\quad \nrm{s^{-1}u^{s}}_{L^\infty} \le C \nrm{s^{-1}w}_{L^{4,1}}. 
					\end{split}
				\end{equation}
			\end{lem}

			\begin{proof} Let us prove the estimate of $\nrm{u^{r}}_{L^\infty}$ and $\nrm{r^{-1}u^{r}}_{L^\infty}$. %in the case $i = 1$.
				We consider the kernel $F^{r}$: \begin{equation*}
					\begin{split}
						-2\pi^2 F^{r} = \br{r}\br{s}\int_0^{2\pi}\int_0^{2\pi} \frac{\cos\br{\tht} (\br{s} - s\cos\br{\phi})}{|x-\bar{x}|^4}  d\br{\tht} d\br{\phi} . 
					\end{split}
				\end{equation*} To prove \eqref{eq:vel-Lorentz}, we just note that \begin{equation*}
					\begin{split}
						|F^{r}| \le C\int_0^{2\pi}\int_0^{2\pi}  \frac{\br{r}\br{s}}{|x-\bar{x}|^3} d\br{\tht} d\br{\phi} . 
					\end{split}
				\end{equation*} which follows from $\br{s} - s\cos\br{\phi}\le |x-\bar{x}|$. Then \begin{equation*}
					\begin{split}
						|u^{r}|\le C\int_{\bbR^4} \frac{1}{|x-\bar{x}|^3} |w| d\bar{x} \le C \nrm{w}_{L^{4,1}}
					\end{split}
				\end{equation*} since $\bar{x}\mapsto \frac{1}{|x-\bar{x}|^3}$ belongs to $L^{\frac{4}{3},\infty}$ for any $x$. 
				
				To prove \eqref{eq:velratio-Lorentz}, we return to the expression for $F^{r}$ and split the $\br{\tht}$-integral into $[-\frac{\pi}{2},\frac{\pi}{2})$ and $[\frac{\pi}{2},\frac{3\pi}{2})$, and then make a change of variable $\br{\tht} \to \br{\tht} - \pi$ in the second region. Defining $\hat{x}$ to be the point corresponding to $(r,\pi,s,0)$, we have that \begin{equation*}
					\begin{split}
						-2\pi^2 F^{r} = 2\br{r}\br{s}\int_0^{2\pi}\int_{-\frac{\pi}{2}}^{\frac{\pi}{2}} \frac{ r\br{r} \cos^{2}\br{\tht}  (\br{s} - s\cos\br{\phi})}{|x-\bar{x}|^4|\hat{x}-\bar{x}|^4} \left( |x-\bar{x}|^2 + |\hat{x}-\bar{x}|^2 \right) d\br{\tht} d\br{\phi} . 
					\end{split}
				\end{equation*} This gives \begin{equation*}
					\begin{split}
						u^{r}(r,s) = c \int_{ \bbR^4 \cap\{ -\frac{\pi}{2}\le\tht <\frac{\pi}{2}\} } \frac{ r\br{r} \cos^{2}\br{\tht}  (\br{s} - s\cos\br{\phi})}{|x-\bar{x}|^4|\hat{x}-\bar{x}|^4} \left( |x-\bar{x}|^2 + |\hat{x}-\bar{x}|^2 \right) w(\br{r},\br{s}) d\bar{x} . 
					\end{split}
				\end{equation*} First, noting that $|\br{s} - s\cos\br{\phi}| \le |x-\bar{x}|$ and $|x-\bar{x}|\le |\hat{x} - \bar{x}| $, \begin{equation*}
					\begin{split}
						\frac{|u^{r}(r,s)|}{r} \le C \int_{ \bbR^4 \cap\{ -\frac{\pi}{2}\le\tht <\frac{\pi}{2}\} } \frac{  \br{r} }{|x-\bar{x}|^3|\hat{x}-\bar{x}|^2} |w(\br{r},\br{s})| d\bar{x} . 
					\end{split}
				\end{equation*} We now consider two cases: (i) $r \le |x-\bar{x}|$ and (ii) $r > |x-\bar{x}|$. In the case (i), we have \begin{equation*}
					\begin{split}
						\br{r} = \br{r} - r + r \le 2|x-\bar{x}|\le 2|\hat{x}-\bar{x}|
					\end{split}
				\end{equation*} and hence \begin{equation*}
					\begin{split}
						\frac{  \br{r} }{|x-\bar{x}|^3|\hat{x}-\bar{x}|^2} \le \frac{1}{|x-\bar{x}|^3 \br{r}}.
					\end{split}
				\end{equation*} In the case (ii), we have $r^2 \ge r^2 + \br{r}^2 - 2r\br{r}\cos\br{\tht},$ which gives $2r\cos\br{\tht}\ge \br{r}$ and hence  $$\br{r} \le \sqrt{r^2+\br{r}^2+2r\br{r}\cos\br{\tht}} \le |\hat{x}-\bar{x}|,$$ which gives the same bound. Therefore, \begin{equation*}
					\begin{split}
						\frac{|u^{r}(r,s)|}{r} \le C \int_{ \bbR^4 } \frac{1}{|x-\bar{x}|^3 }\frac{w}{\br{r}} d\bar{x} \le C \nrm{r^{-1}w}_{L^{4,1}}
					\end{split}
				\end{equation*} since the function $\bar{x}\mapsto \frac{1}{|x-\bar{x}|^3 }$ belongs to $L^{\frac{4}{3},\infty}$ for any $x$. 
			\end{proof}

			\subsection{Local well-posedness in the no-swirl case
			}
			
			In this section, we prove a well-posedness result for the no-swirl system \eqref{eq_Vorticityform}:
			\begin{equation*}%\label{eq:4D-no-swirl}
				\begin{split}
					\rd_t w +  (u^{r} \rd_{r} +u^{s} \rd_{s}  ) w =  \left( \frac{u^{r}}{r} + \frac{u^{s}}{s}  \right) w,
				\end{split}
			\end{equation*} supplemented with \eqref{eq:4D-div} and \eqref{eq:4D-vort}. In the remainder of this section, we shall write $u = (u^{r},u^{s})$ and $u\cdot\nb = u^{r}\rd_{r} + u^{s}\rd_{s}$. Under the assumption $ w\in L^{4,1}(\bbR^{4}) $ and $ w\in L^{\ift}(\bbR^{4}) $, $u\in \mathrm{logLip}$ and therefore the flow map is well-defined as a measure-preserving homeomorphism of $\bbR^4$. We borrow a proposition from \cite{Danaxi} which works in this setting: 
			\begin{prop}[{{\cite[Prop. 2]{Danaxi}}}]\label{prop:Dan}
				\begin{comment}
					\Blue{Let $ T>0 $, and $ v $ be a divergence-free vector field in $ L^{1}\big(0,T;C_{\ast}^{1}(\bbR^{4})\big) $, where $ C_{\ast}^{1}(\bbR^{4}) $ is the Zygmund class of bounded and continuous functions with the norm
						$$ \nrm{v}_{C_{\ast}^{1}(\bbR^{4})}:=\nrm{v}_{L^{\ift}(\bbR^{4})}+\sup_{x,y\in\bbR^{4}}\frac{|v(x+y)+v(x-y)-2v(x)|}{|y|^{2}}<\ift. $$}
					\Red{(cf. Need to check if $ v\in\text{lopLip} $ implies $ v\in C_{\ast}^{1} $.)} Also 
				\end{comment}
				Let $ u $ be a divergence-free vector field in $ \bbR^{4} $ and 
				$h$ be a solution to \begin{equation*}
					\begin{split}
						\rd_t h+u\cdot\nb h = f
					\end{split}
				\end{equation*} on some time interval $[0,T]$. Then for any $1\le p,q\le+\infty$, \begin{equation*}
					\begin{split}
						\nrm{h(t)}_{L^{p,q}} \le \nrm{h_0}_{L^{p,q}} + \int_0^t \nrm{f(\tau)}_{L^{p,q}} d\tau ,
					\end{split}
				\end{equation*} for any $0<t\le T$. 
			\end{prop}
			\begin{comment}
				\Red{Stating the required condition of $ u $ might be needed, at least the fact that $ u $ is divergence-free in $ \bbR^{4} $. -DL.}
			\end{comment}
			We now give the proof of Theorem \ref{thm_localreg}.
			
			\begin{proof}[Proof of Theorem \ref{thm_localreg}]
				As we are concerned with solutions having uniformly bounded vorticity and velocity, uniqueness of the solution within such class is guaranteed in \cite[Section 3]{Danaxi}. Moreover, once a priori estimates are given, existence of a solution can be proved using standard smoothing procedure. Therefore let us just prove the necessary a priori estimates. 
				From the equation for $w$ and $ w/(rs) $, we have \begin{equation}\label{eq_wLinfty}
					\begin{split}
						\frac{d}{dt} \nrm{w}_{L^\infty} \le (\nrm{r^{-1}u^{r}}_{L^\infty} + \nrm{s^{-1}u^{s}}_{L^\infty}) \nrm{w}_{L^\infty} \le C (\nrm{r^{-1}w}_{L^{4,1}} + \nrm{s^{-1}w}_{L^{4,1}}) \nrm{w}_{L^\infty},\quad\bigg\|\frac{w(t)}{rs}\bigg\|_{L^{4,1}(\bbR^{4})}=\bigg\|\frac{w_{0}}{rs}\bigg\|_{L^{4,1}(\bbR^{4})},
					\end{split}
				\end{equation} and similarly \begin{equation}\label{eq_wL4,1}
					\begin{split}
						\frac{d}{dt} \nrm{w}_{L^{4,1}}  \le C (\nrm{r^{-1}w}_{L^{4,1}} + \nrm{s^{-1}w}_{L^{4,1}}) \nrm{w}_{L^{4,1}}
					\end{split}
				\end{equation} This shows that formally the norms $\nrm{w}_{L^\infty}$ and $\nrm{w}_{L^{4,1}}$ remain finite as long as $\nrm{r^{-1}w}_{L^{4,1}} + \nrm{s^{-1}w}_{L^{4,1}}$ does not blow up. 
				Now consider the equation for $\frac{w}{r}$ and $\frac{w}{s}$: 
				\begin{align}
					\rd_t \frac{w}{r}+  (u^{r} \rd_{r} +u^{s} \rd_{s}  ) \frac{w}{r} &= \frac{u^{s}}{s} \frac{w}{r},\label{eq_woverr}\\
					\rd_t \frac{w}{s}+  (u^{r} \rd_{r} +u^{s} \rd_{s}  ) \frac{w}{s} &= \frac{u^{r}}{r} \frac{w}{s}.\label{eq_wovers}
				\end{align} Estimating similarly as in the above, we have \begin{equation*}
					\begin{split}
						\frac{d}{dt}\nrm{r^{-1}w}_{L^{4,1}} &\le C\nrm{s^{-1}w}_{L^{4,1}}\nrm{r^{-1}w}_{L^{4,1}}, \qquad \frac{d}{dt}\nrm{s^{-1}w}_{L^{4,1}} \le C\nrm{r^{-1}w}_{L^{4,1}}\nrm{s^{-1}w}_{L^{4,1}}.
					\end{split}
				\end{equation*} This gives \begin{equation*}
					\begin{split}
						\frac{d}{dt}(\nrm{r^{-1}w}_{L^{4,1}}+\nrm{s^{-1}w}_{L^{4,1}}) \le C(\nrm{r^{-1}w}_{L^{4,1}}+\nrm{s^{-1}w}_{L^{4,1}})^2. 
					\end{split}
				\end{equation*} Therefore, for some interval of time $[0,T]$ with $T$ depending only on $ \nrm{r^{-1}w_0}_{L^{4,1}} + \nrm{s^{-1}w_0}_{L^{4,1}}$, $\nrm{r^{-1}w}_{L^{4,1}}+\nrm{s^{-1}w}_{L^{4,1}}$ remains bounded. Finally, we note that on the same time interval, $\nrm{u}_{L^\infty} \le C \nrm{w}_{L^{4,1}}$ remains finite as well. 
			\end{proof}
			\begin{comment}
				\Red{As we changed the condition of Theorem \ref{thm_localreg}, its proof above could change as well. More precisely, we are currently working on the system \begin{equation}
						\begin{split}
							\frac{d}{dt}\bigg[\bigg]
						\end{split}
				\end{equation}}
			\end{comment}

			\begin{rmk} 
				We point out that the solution is global in time once $\bbR^4$ is replaced with a bounded domain $\Omg$ satisfying bi-rotational symmetry, for example the $4$-ball defined by $\{ x\in\bbR^4:|x|\le1 \}$. Actually, as we shall see below it is only necessary to be bounded either in $r$ or $s$. One needs to add the condition that $\frac{w_0}{rs} \in L^{4,1}(\bbR^4)$. Then, there exists a unique global-in-time solution $w(t)$ to \eqref{eq_Vorticityform}. To see this, we estimate $r^{-1}w$ %and $s^{-1}w$ 
				as follows: \begin{equation*}
					\begin{split}
						\frac{d}{dt}\left\Vert \frac{w(t)}{r}\right\Vert_{L^{4,1}} &\le  \left\Vert \frac{u^{s}(t) }{s} \frac{w(t)}{r}\right\Vert_{L^{4,1}},%\\
						%\frac{d}{dt}\left\Vert \frac{w(t)}{s}\right\Vert_{L^{4,1}} &\le  \left\Vert \frac{u^{r}(t) }{r} \frac{w(t)}{s}\right\Vert_{L^{4,1}}.
					\end{split}
				\end{equation*} and we bound \begin{equation*}
					\begin{split}
						\left\Vert \frac{u^{s}(t) }{s} \frac{w(t)}{r}\right\Vert_{L^{4,1}}  &= \left\Vert  u^{s}(t)  \frac{w(t)}{rs}\right\Vert_{L^{4,1}}  \le C\nrm{u(t)}_{L^\infty} \left\Vert \frac{w_0}{rs}\right\Vert_{L^{4,1}}. 
					\end{split}
				\end{equation*} On the other hand, \begin{equation*}
					\begin{split}
						\nrm{u(t)}_{L^\infty}\le \nrm{w(t)}_{L^{4,1}} \le CR\left\Vert \frac{w(t)}{r}\right\Vert_{L^{4,1}}
					\end{split}
				\end{equation*} where $R$ is the radius of $\Omg$ in the $r$-direction. Hence   \begin{equation}\label{eq:est-bdd-domain}
					\begin{split}
						\left\Vert \frac{w(t)}{r}\right\Vert_{L^{4,1}} \le \left\Vert \frac{w_0}{r}\right\Vert_{L^{4,1}} \exp\left( CR  \left\Vert \frac{w_0}{rs}\right\Vert_{L^{4,1}} \right).
					\end{split}
				\end{equation} Then, it is easy to see that $\nrm{s^{-1}w}_{L^{4,1}}$ and therefore $\nrm{w}_{L^\infty}$ remains bounded as well. 
				Of course, the estimates for the velocity needs to be re-proved since the domain is not $\bbR^4$ anymore. This can be done following \cite{Danaxi,Tam-axi} and we omit the details. 
			\end{rmk} 
			
			\begin{comment}
				One may try to consider the $\bbR^4$ case with compactly supported initial $w_0$ with an additional assumption on $\frac{w_0}{rs}$. Then, the solution will not blow up unless the radius of the support $R(t)$ remains finite, but it seems that the best estimate we can do using the conserved quantities is simply $\frac{d}{dt}R(t)\lesssim (R(t))^2$. 
			\end{comment}
			\begin{comment}
				Up to now was about the derivation of the Biot--Savart law and the local well-posedness result of the Yudovich-type solution of \eqref{eq_Vorticityform} from Prof. In-Jee Jeong's note. The contents of Section \ref{sec_GWP} is about the global well-posedness of the Yudovich-type solution with certain conditions, which is the most recently obtained result.
			\end{comment}
			
			\section{Global regularity of $ w $}\label{sec_GWP}
			
			%\subsection{Kernels of $ u^{r} $ and $ u^{s} $}\label{subsec_kernel}
			
			In this section, we present key %some technical 
			lemmas and estimates to prove Theorem \ref{thm_globreg}.

			\begin{comment}
				$$ X_{--}=X_{--}(r,s,\br{r},\br{s},\br{\tht},\br{\phi}):=(r-\br{r})^{2}+(s-\br{s})^{2}+2r\br{r}(1-\cos\br{\tht})+2s\br{s}(1-\cos\br{\phi})%r^{2}+\br{r}^{2}-2r\br{r}\cos\tht+s^{2}+\br{s}^{2}-2s\br{s}\cos\phi
				. $$
			\end{comment}

			\subsection{Some technical lemma}\label{subsec_techlem}
			Firstly, we introduce the following %a technical 
			lemma which is frequently used in the proof of Theorem \ref{thm_globreg}. 
			We denote
			\begin{equation}\label{eq_ftau}
				f_{a}(\tau):=\int_{0}^{\pi}\frac{1}{[2(1-\cos\tht)+\tau]^{a}}d\tht,\quad \tau>0,\ a>\frac{1}{2}.
			\end{equation}
			
			\begin{lem}\label{lem_estoff}
				$f_{a}$ from \eqref{eq_ftau} satisfies
				\begin{equation}\label{eq_ftauest}
					f_{a}(\tau)\lesssim_{a}\min\bigg\lbrace\frac{1}{\tau^{a-1/2}},\frac{1}{\tau^{a}}\bigg\rbrace,\quad\tau>0.
				\end{equation}
			\end{lem}
			
			\begin{proof}
				%We follow the same procedure as in \cite[Lem. 2.2]{CJLglobal22}. First, we prove
				The upper bound $f_{a}(\tau)\lesssim \tau^{-a}$ immediately follows from the form of $f_{a}$, so it suffices to show
				$$ \tau^{a-1/2}f_{a}(\tau)\lesssim_{a}1,\quad 0<\tau<1. $$
				Before proving this, let us take a sufficiently small $\veps_{0}>0$ that satisfies
				$$ 2(1-\cos\tht)\geq \frac{\tht^{2}}{2},\quad \tht\in[0,\veps_{0}]. $$
				We fix $\tau\in(0,1)$ and split $\tau^{a-1/2} f_{a}(\tau)$ into
				$$ \tau^{a-1/2} f_{a}(\tau)=\underbrace{\int_{0}^{\veps_{0}}\frac{\tau^{a-1/2}}{[2(1-\cos\tht)+\tau]^{a}}d\tht}_{=(A)}+\underbrace{\int_{\veps_{0}}^{\pi}\frac{\tau^{a-1/2}}{[2(1-\cos\tht)+\tau]^{a}}d\tht}_{=(B)}. $$
				For $(A)$, we use the change of variables $\alp=\frac{\tht}{\sqrt{2\tau}}$ to get
				$$ (A)\simeq\int_{0}^{\frac{\veps_{0}}{\sqrt{2\tau}}}\frac{1}{(\alp^{2}+1)^{a}}d\alp<\ift. $$
				For $(B)$, we use the fact that $2(1-\cos\tht)\geq C$ for some $C>0$ when $\tht\in[\veps_{0},\pi]$ to get
				$$ (B)\leq\int_{\veps_{0}}^{\pi}\frac{\tau}{C^{a}}d\tht<\ift. $$
				\begin{comment}
					Now let us show that we have
					$$ |f_{a}(\tau)|\lesssim\frac{1}{\tau^{}},\quad \tau\geq M, $$
					for some $M\geq1$. We let $\xi:=\frac{1}{\tau}$ and define $g(\xi):=f(\frac{1}{\xi})$, which becomes
					$$ g(\xi)=\int_{0}^{\frac{\pi}{2}}\frac{\xi^{\frac{3}{2}}}{[2\xi(1-\cos\tht)+1]^{\frac{3}{2}}}d\tht. $$
					Then using the expansion
					$$ \frac{1}{(x+1)^{\frac{3}{2}}}=1-\frac{3}{2}x+O(x^{2}), $$
					we have
					$$ g(\xi)=\xi^{\frac{3}{2}}\bigg[\int_{0}^{\frac{\pi}{2}}d\tht-\frac{3}{2}\cdot2\xi\int_{0}^{\frac{\pi}{2}}(1-\cos\tht)d\tht+O(\xi^{2})\bigg]\simeq\xi^{\frac{3}{2}}+O(\xi^{\frac{5}{2}}). $$
					
				\end{comment}
				Thus, we get $ f_{a}(\tau)%{\tiny }=g\bigg(\frac{1}{\tau}\bigg)%\bigg|\bigg|
				\lesssim_{a}1/\tau^{a-1/2},$ which finishes the proof.
			\end{proof}
			
			\subsection{Estimates of $\nrm{w(t)}_{L^{\ift}(\bbR^{4})}$, %$\nrm{w(t)}_{L^{1}(\bbR^{4})}$, 
				$\nrm{rs^{-1}w(t)}_{L^{1}(\bbR^{4})}$, and $\nrm{r^{-1}sw(t)}_{L^{1}(\bbR^{4})}$}
			In this subsection, we will show %some lemmas about 
			estimates of terms $\nrm{w(t)}_{L^{\ift}(\bbR^{4})}$, $\nrm{rs^{-1}w(t)}_{L^{1}(\bbR^{4})}$, and $\nrm{r^{-1}sw(t)}_{L^{1}(\bbR^{4})}$ in time. The following lemma says that if $ w_{0} $ has decay near the axes $ r=0 $ and $ s=0 $, then the supremum of $ w(t) $ is bounded by the ``length" function $ L $ defined in \eqref{eq_Ltdef}.
			\begin{lem}\label{lem_wtest}
				Let $w_{0}$ satisfy $%(1+r)(1+s)w_{0}%
				(1+r+s+rs)w_{0}/(rs)%w_{0},\frac{w_{0}}{r},\frac{w_{0}}{s},\frac{w_{0}}{rs}
				\in L^{\ift}(\bbR^{4})$. Then we have
				\begin{equation}\label{eq_wtest}
					\nrm{w(t)}_{L^{\ift}(\bbR^{4})}\lesssim\bigg[\nrm{w_{0}}_{L^{\ift}(\bbR^{4})}+\bigg\|\frac{w_{0}}{r}\bigg\|_{L^{\ift}(\bbR^{4})}+\bigg\|\frac{w_{0}}{s}\bigg\|_{L^{\ift}(\bbR^{4})}+\bigg\|\frac{w_{0}}{rs}\bigg\|_{L^{\ift}(\bbR^{4})}\bigg]\cdot\bigg[1+\int_{0}^{t}(\nrm{u^{r}(\tau)}_{L^{\ift}(\bbR^{4})}+\nrm{u^{s}(\tau)}_{L^{\ift}(\bbR^{4})})d\tau\bigg]^{2}.
				\end{equation}
			\end{lem}
			
			\begin{comment}
				First, let us consider the case $|x|\leq L(t)$. Then we have
				$$ |w(t,x)|=r_{x}s_{x}\cdot\frac{|w(t,x)|}{r_{x}s_{x}}\lesssim (r_{x}^{2}+s_{x}^{2})\cdot\bigg\|\frac{w(t)}{rs}\bigg\|_{L^{\ift}(\bbR^{4})}=|x|^{2}\cdot\bigg\|\frac{w_{0}}{rs}\bigg\|_{L^{\ift}(\bbR^{4})}\leq L(t)^{2}\cdot\bigg\|\frac{w_{0}}{rs}\bigg\|_{L^{\ift}(\bbR^{4})}. $$
				Now let us consider the other case $|x|>L(t)$. For this case, 
			\end{comment}
			\begin{proof}
				We fix $t\geq0$ and define
				\begin{equation}\label{eq_Ltdef}
					L(t):=1+\int_{0}^{t}[\nrm{u^{r}(\tau)}_{L^{\ift}(\bbR^{4})}+\nrm{u^{s}(\tau)}_{L^{\ift}(\bbR^{4})}]d\tau.
				\end{equation}
				In addition, we denote $\Phi_{t}$ as the flow map that solves the ODE
				$$ \frac{d}{dt}\Phi_{t}(x)=u(t,\Phi_{t}(x)),\quad \Phi_{0}(x)=x. $$
				Also, we use the notation $r_{x}:=\sqrt{x_{1}^{2}+x_{1}^{2}}$, $s_{x}:=\sqrt{x_{3}^{2}+x_{4}^{2}}$, and denote $\Phi_{t}^{r}$ and $\Phi_{t}^{s}$ as the $r$ and $s$-component of $\Phi_{t}$, respectively.
				
				\medskip
				\noindent We take $y=\Phi_{t}^{-1}(x)$, where $\Phi_{t}^{-1}$ is the inverse of $\Phi_{t}$. Then note that we have
				$$ x=\Phi_{t}(y)=\Phi_{0}(y)+\int_{0}^{t}u(\tau,\Phi_{\tau}(y))d\tau=y+\int_{0}^{t}u(\tau,\Phi_{\tau}(y))d\tau, $$
				which gives us
				\begin{equation*}
					\begin{split}
						r_{x}&=r_{y}+\int_{0}^{t}u^{r}(\tau,\Phi_{\tau}(y))d\tau,\quad s_{x}=s_{y}+\int_{0}^{t}u^{s}(\tau,\Phi_{\tau}(y))d\tau.
					\end{split}
				\end{equation*}
				Also, we use the conservation of the relative vorticity $\frac{w}{rs}$ in time along the flow $\Phi_{t}$:
				$$ \frac{w(t,x)}{r_{x}s_{x}}=\frac{w(t,\Phi_{t}(y))}{\Phi_{t}^{r}(y)\Phi_{t}^{s}(y)}=\frac{w_{0}(y)}{r_{y}s_{y}}. $$
				Then we have
				\begin{equation*}
					\begin{split}
						|w(t,x)|&=r_{x}s_{x}\cdot\frac{|w(t,x)|}{r_{x}s_{x}}=\frac{r_{x}s_{x}}{r_{y}s_{y}}\cdot|w_{0}(y)|=\bigg(1+\frac{1}{r_{y}}\int_{0}^{t}u^{r}(\tau,\Phi_{\tau}(y))d\tau\bigg)\bigg(1+\frac{1}{s_{y}}\int_{0}^{t}u^{s}(\tau,\Phi_{\tau}(y))d\tau\bigg)\cdot|w_{0}(y)|\\
						&\leq\bigg(1+\frac{1}{r_{y}}\int_{0}^{t}\nrm{u^{r}(\tau)}_{L^{\ift}(\bbR^{4})}d\tau\bigg)\bigg(1+\frac{1}{s_{y}}\int_{0}^{t}\nrm{u^{s}(\tau)}_{L^{\ift}(\bbR^{4})}d\tau\bigg)\cdot|w_{0}(y)|\\
						&\leq\nrm{w_{0}}_{L^{\ift}(\bbR^{4})}+\bigg(\bigg\|\frac{w_{0}}{r}\bigg\|_{L^{\ift}(\bbR^{4})}+\bigg\|\frac{w_{0}}{s}\bigg\|_{L^{\ift}(\bbR^{4})}\bigg)\cdot L(t)+\bigg\|\frac{w_{0}}{rs}\bigg\|_{L^{\ift}(\bbR^{4})}\cdot L(t)^{2}.
					\end{split}
				\end{equation*}
			\end{proof}

			Next, the lemma below says that if the relative vorticity $ w_{0}/(rs) $ has a weak decay at infinity, then the $ L^{1}- $norms of $ rs^{-1}w(t) $ and $ r^{-1}sw(t) $ are controlled by the length function $ L $ from \eqref{eq_Ltdef}.
			
			\begin{lem}\label{lem_wL1}
				Let $ w_{0} $ %$\omg_{0}$ 
				satisfy $%w_{0}, 
				%\frac{sw_{0}}{r}, \frac{rw_{0}}{s},%\frac{w_{0}}{r},\frac{w_{0}}{s},
				(1+r^{2}+s^{2})w_{0}/(rs)%\frac{w_{0}}{rs}
				\in L^{1}(\bbR^{4})$. %$ w_{0}, rw_{0}, sw_{0}, \frac{w_{0}}{r}, \frac{w_{0}}{s}, \frac{sw_{0}}{r}, \frac{rw_{0}}{s}, \frac{w_{0}}{rs}\in L^{1}(\bbR^{4}) $. 
				Then for any $t\geq0$, we have
				\begin{align}
					%\nrm{w(t)}_{L^{1}(\bbR^{4})}&\lesssim\bigg[\nrm{w_{0}}_{L^{1}(\bbR^{4})}+\bigg\|\frac{w_{0}}{r}\bigg\|_{L^{1}(\bbR^{4})}+\bigg\|\frac{w_{0}}{s}\bigg\|_{L^{1}(\bbR^{4})}+\bigg\|\frac{w_{0}}{rs}\bigg\|_{L^{1}(\bbR^{4})}\bigg]\cdot\bigg[1+\int_{0}^{t}(\nrm{u^{r}(\tau)}_{L^{\ift}(\bbR^{4})}+\nrm{u^{s}(\tau)}_{L^{\ift}(\bbR^{4})})d\tau\bigg]^{2},\label{eq_wL1est}\\
					\bigg\|\frac{rw(t)}{s}\bigg\|_{L^{1}(\bbR^{4})}&\lesssim\bigg[\bigg\|\frac{rw_{0}}{s}\bigg\|_{L^{1}(\bbR^{4})}%+\bigg\|\frac{w_{0}}{s}\bigg\|_{L^{1}(\bbR^{4})}
					+\bigg\|\frac{w_{0}}{rs}\bigg\|_{L^{1}(\bbR^{4})}\bigg]\cdot\bigg[1+\int_{0}^{t}(\nrm{u^{r}(\tau)}_{L^{\ift}(\bbR^{4})}+\nrm{u^{s}(\tau)}_{L^{\ift}(\bbR^{4})})d\tau\bigg]^{2},\label{eq_rwoversL1est}\\
					\bigg\|\frac{sw(t)}{r}\bigg\|_{L^{1}(\bbR^{4})}&\lesssim\bigg[\bigg\|\frac{sw_{0}}{r}\bigg\|_{L^{1}(\bbR^{4})}%+\bigg\|\frac{w_{0}}{r}\bigg\|_{L^{1}(\bbR^{4})}
					+\bigg\|\frac{w_{0}}{rs}\bigg\|_{L^{1}(\bbR^{4})}\bigg]\cdot\bigg[1+\int_{0}^{t}(\nrm{u^{r}(\tau)}_{L^{\ift}(\bbR^{4})}+\nrm{u^{s}(\tau)}_{L^{\ift}(\bbR^{4})})d\tau\bigg]^{2}.\label{eq_swoverrL1est}
				\end{align}
				\begin{comment}
					\begin{split}
						\nrm{w(t)}_{L^{1}(\bbR^{4})}\lesssim\bigg[\nrm{w_{0}}_{L^{1}(\bbR^{4})}+\bigg\|\frac{w_{0}}{r}\bigg\|_{L^{1}(\bbR^{4})}+\bigg\|\frac{w_{0}}{s}\bigg\|_{L^{1}(\bbR^{4})}+\bigg\|\frac{w_{0}}{rs}\bigg\|_{L^{1}(\bbR^{4})}\bigg]\cdot\bigg[1+\int_{0}^{t}(\nrm{u^{r}(\tau)}_{L^{\ift}(\bbR^{4})}+\nrm{u^{s}(\tau)}_{L^{\ift}(\bbR^{4})})d\tau\bigg]^{2}.%\bigg[1&+\nrm{w_{0}}_{L^{1}(\bbR^{4})}+\nrm{rw_{0}}_{L^{1}(\bbR^{4})}+\nrm{sw_{0}}_{L^{1}(\bbR^{4})}\\
						%&+\bigg\|\frac{w_{0}}{r}\bigg\|_{L^{1}(\bbR^{4})}+\bigg\|\frac{w_{0}}{s}\bigg\|_{L^{1}(\bbR^{4})}+\bigg\|\frac{sw_{0}}{r}\bigg\|_{L^{1}(\bbR^{4})}+\bigg\|\frac{rw_{0}}{s}\bigg\|_{L^{1}(\bbR^{4})}+\bigg\|\frac{w_{0}}{rs}\bigg\|_{L^{1}(\bbR^{4})}\bigg]\cdot L(t)^{4}.
					\end{split}
				\end{comment}
			\end{lem}
			
			\begin{comment}
				\Blue{Previously, there was an estimate of the term $ \nrm{w}_{L^{1}(\bbR^{4})} $. However, 
					since $ (sw)/r, (rw)/s\in L^{1}(\bbR^{4}) $ implies $ w\in L^{1}(\bbR^{4}) $, we do not need the estimate of the term $ \nrm{w}_{L^{1}(\bbR^{4})} $. In the rest of the below, the blue-colored term $ \nrm{(rw)/s}_{L^{1}(\bbR^{4})}+\nrm{(sw)/r}_{L^{1}(\bbR^{4})} $ means that we removed the term $ \nrm{w}_{L^{1}(\bbR^{4})} $ that was originally added there. 
					-DL.}
			\end{comment}
			
			\begin{proof}[Proof of Lemma \ref{lem_wL1}]
				We use the same notations from the proof of the previous lemma. 
				\begin{comment}
					First, we have
					\begin{equation*}
						\begin{split}
							\nrm{w(t)}_{L^{1}(\bbR^{4})}&=\int_{\bbR^{4}}%r_{x}s_{x}
							|w(t,x)|dx=\int_{\bbR^{4}}r_{x}s_{x}\cdot\frac{|w(t,x)|}{r_{x}s_{x}}dx=\int_{\bbR^{4}}\Phi_{t}^{r}(y)\Phi_{t}^{s}(y)\cdot\frac{|w(t,\Phi_{t}(y))|}{\Phi_{t}^{r}(y)\Phi_{t}^{s}(y)}dy\\
							&=\int_{\bbR^{4}}\Phi_{t}^{r}(y)\Phi_{t}^{s}(y)\cdot\frac{|w_{0}(y)|}{r_{y}s_{y}}dy=\int_{\bbR^{4}}\frac{\Phi_{t}^{r}(y)}{r_{y}}\cdot\frac{\Phi_{t}^{s}(y)}{s_{y}}\cdot |w_{0}(y)|dy\\
							&\leq\int_{\bbR^{4}}\bigg[1+\frac{L(t)}{r_{y}}\bigg]\bigg[1+\frac{L(t)}{s_{y}}\bigg]\cdot |w_{0}(y)|dy.
						\end{split}
					\end{equation*}
					Computing the last term above, we get
					\begin{equation*}
						\begin{split}
							\nrm{w(t)}_{L^{1}(\bbR^{4})}\leq\nrm{w_{0}}_{L^{1}(\bbR^{4})}+\bigg(\bigg\|\frac{w_{0}}{r}\bigg\|_{L^{1}(\bbR^{4})}+\bigg\|\frac{w_{0}}{s}\bigg\|_{L^{1}(\bbR^{4})}\bigg)\cdot L(t)+\bigg\|\frac{w_{0}}{rs}\bigg\|_{L^{1}(\bbR^{4})}\cdot L(t)^{2}.%&\ \bigg\|\frac{w_{0}}{rs}\bigg\|_{L^{1}(\bbR^{4})}\cdot L(t)^{4}+2\bigg[\bigg\|\frac{w_{0}}{r}\bigg\|_{L^{1}(\bbR^{4})}+\bigg\|\frac{w_{0}}{s}\bigg\|_{L^{1}(\bbR^{4})}\bigg]\cdot L(t)^{3}\\
							%&+\bigg[4\nrm{w_{0}}_{L^{1}(\bbR^{4})}+\bigg\|\frac{sw_{0}}{r}\bigg\|_{L^{1}(\bbR^{4})}+\bigg\|\frac{rw_{0}}{s}\bigg\|_{L^{1}(\bbR^{4})}\bigg]\cdot L(t)^{2}+\bigg[2\nrm{rw_{0}}_{L^{1}(\bbR^{4})}+2\nrm{sw_{0}}_{L^{1}(\bbR^{4})}+1\bigg]\cdot L(t).
						\end{split}
					\end{equation*}
				\end{comment}
				To prove \eqref{eq_rwoversL1est}, let us consider the case when $r_{x}\leq L(t)$ or when $r_{x}>L(t)$. For the first case, we have
				\begin{equation*}
					\begin{split}
						\bigg\|\frac{rw(t)}{s}\bigg\|_{L^{1}(\bbR^{4})}&=\int_{\bbR^{4}}\frac{r_{x}|w(t,x)|}{s_{x}}dx=\int_{\bbR^{4}}r_{x}^{2}%s_{x}
						\cdot\frac{|w(t,x)|}{r_{x}s_{x}}dx=\int_{\bbR^{4}}[\Phi_{t}^{r}(y)]^{2}%\Phi_{t}^{s}(y)
						\cdot\frac{|w(t,\Phi_{t}(y))|}{\Phi_{t}^{r}(y)\Phi_{t}^{s}(y)}dy\\
						&=\int_{\bbR^{4}}[\Phi_{t}^{r}(y)]^{2}%\Phi_{t}^{s}(y)
						\cdot\frac{|w_{0}(y)|}{r_{y}s_{y}}dy\leq\bigg\|\frac{w_{0}}{rs}\bigg\|_{L^{1}(\bbR^{4})}\cdot L(t)^{2}.
					\end{split}
				\end{equation*}
				For the other case, we get $r_{y}>1$, which gives us
				\begin{equation*}
					\begin{split}
						\bigg\|\frac{rw(t)}{s}\bigg\|_{L^{1}(\bbR^{4})}%&=%\int_{\bbR^{4}}\frac{r_{x}|w(t,x)|}{s_{x}}dx=\int_{\bbR^{4}}r_{x}^{2}%s_{x}
						%\cdot\frac{|w(t,x)|}{r_{x}s_{x}}dx=\int_{\bbR^{4}}[\Phi_{t}^{r}(y)]^{2}%\Phi_{t}^{s}(y)
						%\cdot\frac{|w(t,\Phi_{t}(y))|}{\Phi_{t}^{r}(y)\Phi_{t}^{s}(y)}dy\\
						&=\int_{\bbR^{4}}[\Phi_{t}^{r}(y)]^{2}%\Phi_{t}^{s}(y)
						\cdot\frac{|w_{0}(y)|}{r_{y}s_{y}}dy%=\int_{\bbR^{4}}\frac{\Phi_{t}^{r}(y)}{r_{y}}\cdot\frac{\Phi_{t}^{s}(y)}{s_{y}}\cdot |w_{0}(y)|dy\\
						%&
						\leq\int_{\bbR^{4}}\bigg[1+\frac{L(t)}{r_{y}}\bigg]^{2}%\bigg[1+\frac{L(t)}{s_{y}}\bigg]
						\cdot \frac{r_{y}|w_{0}(y)|}{s_{y}}dy\lesssim\bigg\|\frac{rw_{0}}{s}\bigg\|_{L^{1}(\bbR^{4})}\cdot L(t)^{2}.
						%&\lesssim\bigg\|\frac{rw_{0}}{s}\bigg\|_{L^{1}(\bbR^{4})}+2\cdot\bigg\|\frac{w_{0}}{s}\bigg\|_{L^{1}(\bbR^{4})}\cdot L(t)+\bigg\|\frac{w_{0}}{rs}\bigg\|_{L^{1}(\bbR^{4})}\cdot L(t)^{2}.
					\end{split}
				\end{equation*}
				Combining the above two leads to \eqref{eq_rwoversL1est}. The proof of \eqref{eq_swoverrL1est} is done in the exact same way.
				%Finally, we have
				\begin{comment}
					\begin{split}
						\bigg\|\frac{sw(t)}{r}\bigg\|_{L^{1}(\bbR^{4})}&=\int_{\bbR^{4}}\frac{s_{x}|w(t,x)|}{r_{x}}dx=\int_{\bbR^{4}}s_{x}^{2}%s_{x}
						\cdot\frac{|w(t,x)|}{r_{x}s_{x}}dx=\int_{\bbR^{4}}[\Phi_{t}^{s}(y)]^{2}%\Phi_{t}^{s}(y)
						\cdot\frac{|w(t,\Phi_{t}(y))|}{\Phi_{t}^{r}(y)\Phi_{t}^{s}(y)}dy\\
						&=\int_{\bbR^{4}}[\Phi_{t}^{s}(y)]^{2}%\Phi_{t}^{s}(y)
						\cdot\frac{|w_{0}(y)|}{r_{y}s_{y}}dy%=\int_{\bbR^{4}}\frac{\Phi_{t}^{r}(y)}{r_{y}}\cdot\frac{\Phi_{t}^{s}(y)}{s_{y}}\cdot |w_{0}(y)|dy\\
						%&
						\leq\int_{\bbR^{4}}\bigg[1+\frac{L(t)}{s_{y}}\bigg]^{2}%\bigg[1+\frac{L(t)}{s_{y}}\bigg]
						\cdot \frac{s_{y}|w_{0}(y)|}{r_{y}}dy\\
						&\lesssim\bigg\|\frac{sw_{0}}{r}\bigg\|_{L^{1}(\bbR^{4})}+2\cdot\bigg\|\frac{w_{0}}{r}\bigg\|_{L^{1}(\bbR^{4})}\cdot L(t)+\bigg\|\frac{w_{0}}{rs}\bigg\|_{L^{1}(\bbR^{4})}\cdot L(t)^{2}.
					\end{split}
				\end{comment}
			\end{proof}
			
			\subsection{Estimate of $u^{r}$ and $u^{s}$}\label{subsec_est}
			Now let us present the key proposition of this section, which is about a time-independent estimate of the supremum of the velocity field $ u $.
			\begin{prop}\label{prop_urLiftest}
				%For any $t\geq0$, we have
				We have
				\begin{equation}\label{eq_urLiftest}
					\max\lbrace\nrm{u^{r}}_{L^{\ift}(\bbR^{4})},\nrm{u^{s}}_{L^{\ift}(\bbR^{4})}\rbrace\lesssim\bigg[%\nrm{w}_{L^{1}(\bbR^{4})}+
					\bigg\|\frac{rw}{s}\bigg\|_{L^{1}(\bbR^{4})}+\bigg\|\frac{sw}{r}\bigg\|_{L^{1}(\bbR^{4})}\bigg]^{1/2}\cdot
					%\bigg[\nrm{w}_{L^{1}(\bbR^{4})}^{1/2}+\bigg\|\frac{rw}{s}\bigg\|_{L^{1}(\bbR^{4})}^{1/2}+\bigg\|\frac{sw}{r}\bigg\|_{L^{1}(\bbR^{4})}^{1/2}\bigg]\cdot
					%\bigg\|\frac{w_{0}}{rs}\bigg\|_{L^{1}(\bbR^{4})}^{1/4}
					\bigg\|\frac{w}{rs}\bigg\|_{L^{\ift}(\bbR^{4})}^{1/2}.
				\end{equation}
			\end{prop}
			
			\begin{proof}
				For any $\lmb>0$, let us consider $\tld{u}^{r}, \tld{u}^{s}$, and $\tld{w}$ defined as
				\begin{equation*}
					\tld{u}^{r}(r,s):=u^{r}(\lmb r,\lmb s),\quad\tld{u}^{s}(r,s):=u^{s}(\lmb r,\lmb s),\quad\tld{w}(r,s):=\lmb w(\lmb r,\lmb s),\quad (r,s)\in\Pi=[0,\ift)^{2},%\lbrace(\br{r},\br{s})\in\bbR^{2} : \br{r},\br{s}\geq0\rbrace,
				\end{equation*}
				where $\tld{u}^{r}$ (resp. $\tld{u}^{s}$) and $\tld{w}$ satisfy the Biot--Savart law \eqref{eq_urform} (resp. \eqref{eq_usform}) as well. Then to show \eqref{eq_urLiftest}, 
				%Thanks to the scaling property, 
				it suffices to prove that for any $(r,s)\in\Pi$ satisfying $r^{2}+s^{2}=1$, we have
				\begin{equation}\label{eq_urscale}
					\max\lbrace|u^{r}(r,s)|,|u^{s}(r,s)|\rbrace\lesssim[%\nrm{rsw}_{L^{1}(\Pi)}+
					\nrm{r^{2}w}_{L^{1}(\Pi)}+\nrm{s^{2}w}_{L^{1}(\Pi)}]^{1/2}\cdot
					%[\nrm{rsw}_{L^{1}(\Pi)}^{1/2}+\nrm{r^{2}w}_{L^{1}(\Pi)}^{1/2}+\nrm{s^{2}w}_{L^{1}(\Pi)}^{1/2}]\cdot
					%\nrm{r^{2}s^{2}w}_{L^{1}(\Pi)}^{1/2}\nrm{w}_{L^{1}(\Pi)}^{1/4}
					\bigg\|\frac{w}{rs}\bigg\|_{L^{\ift}(\Pi)}^{1/2}.
				\end{equation}
				Indeed, if \eqref{eq_urscale} is true, then the same estimate holds for $\tld{u}^{r}$, $\tld{u}^{s}$, and $\tld{w}$ as well:
				\begin{equation*}
					\begin{split}
						|\tld{u}^{r}(r,s)|=|u^{r}(\lmb r,\lmb s)|,&\quad|\tld{u}^{s}(r,s)|=|u^{s}(\lmb r,\lmb s)|,\\
						%\max&\lbrace|\tld{u}^{r}(r,s)|,|\tld{u}^{s}(r,s)|\rbrace=\max\lbrace|u^{r}(\lmb r,\lmb s)|,|u^{s}(\lmb r,\lmb s)|\rbrace,\\
						[%\nrm{rs\tld{w}}_{L^{1}(\Pi)}%+
						\nrm{r^{2}\tld{w}}_{L^{1}(\Pi)}+\nrm{s^{2}\tld{w}}_{L^{1}(\Pi)}]^{1/2}\cdot\bigg\|\frac{\tld{w}}{rs}\bigg\|_{L^{\ift}(\Pi)}^{1/2}
						%[\nrm{rs\tld{w}}_{L^{1}(\Pi)}^{1/2}&+\nrm{r^{2}\tld{w}}_{L^{1}(\Pi)}^{1/2}+\nrm{s^{2}\tld{w}}_{L^{1}(\Pi)}^{1/2}]\cdot\bigg\|\frac{\tld{w}}{rs}\bigg\|_{L^{\ift}(\Pi)}^{1/2}\\
						&=[%\lmb^{-3}\nrm{rsw}_{L^{1}(\Pi)}+
						\lmb^{-3}\nrm{r^{2}w}_{L^{1}(\Pi)}+\lmb^{-3}\nrm{s^{2}w}_{L^{1}(\Pi)}]^{1/2}\cdot\lmb^{3/2}\bigg\|\frac{w}{rs}\bigg\|_{L^{\ift}(\Pi)}^{1/2}\\
						&=[%\nrm{rsw}_{L^{1}(\Pi)}+
						\nrm{r^{2}w}_{L^{1}(\Pi)}+\nrm{s^{2}w}_{L^{1}(\Pi)}]^{1/2}\cdot\bigg\|\frac{w}{rs}\bigg\|_{L^{\ift}(\Pi)}^{1/2}.
						%&=[\nrm{rsw}_{L^{1}(\Pi)}^{1/2}+\nrm{r^{2}w}_{L^{1}(\Pi)}^{1/2}+\nrm{s^{2}w}_{L^{1}(\Pi)}^{1/2}]\cdot\bigg\|\frac{w}{rs}\bigg\|_{L^{\ift}(\Pi)}^{1/2}.
					\end{split}
				\end{equation*}
				The proof is completed once we take supremum with respect to $\lmb>0$ on the left-hand side.
				
				\medskip
				\noindent Throughout the proof, %For convenience, 
				let us use the following notations%denote
				\begin{equation*}
					\begin{split}
						X_{--}&=X_{--}(r,s,\br{r},\br{s},\br{\tht},\br{\phi}):=(r-\br{r})^{2}+(s-\br{s})^{2}+2r\br{r}(1-\cos\br{\tht})+2s\br{s}(1-\cos\br{\phi}),\\
						X_{+-}&=X_{+-}(r,s,\br{r},\br{s},\br{\tht},\br{\phi}):=(r-\br{r})^{2}+(s-\br{s})^{2}+2r\br{r}(1+\cos\br{\tht})+2s\br{s}(1-\cos\br{\phi}),\\
						X_{-+}&=X_{-+}(r,s,\br{r},\br{s},\br{\tht},\br{\phi}):=(r-\br{r})^{2}+(s-\br{s})^{2}+2r\br{r}(1-\cos\br{\tht})+2s\br{s}(1+\cos\br{\phi}),\\
						X_{++}&=X_{++}(r,s,\br{r},\br{s},\br{\tht},\br{\phi}):=(r-\br{r})^{2}+(s-\br{s})^{2}+2r\br{r}(1+\cos\br{\tht})+2s\br{s}(1+\cos\br{\phi}).
					\end{split}
				\end{equation*}
				%for notational convenience. 
				%Throughout the proof, 
				We shall frequently use the inequalities
				$$ X_{--}\leq \min\lbrace X_{+-}, X_{-+}\rbrace,\quad \max\lbrace X_{+-}, X_{-+}\rbrace\leq X_{++}, $$
				
				and the simple estimates \begin{equation*}
					\begin{split}
						|s-\br{s}|,\quad  |r-\br{r}| \lesssim X_{--}^{\frac12}, \qquad r\br{r}(1-\cos\br{\tht}), \quad  s\br{s}(1-\cos\br{\phi}) \lesssim X_{--}. 
					\end{split}
				\end{equation*} This holds since each of the four terms in $X_{--}=(r-\br{r})^{2}+(s-\br{s})^{2}+2r\br{r}(1-\cos\br{\tht})+2s\br{s}(1-\cos\br{\phi})$ is nonnegative. Moreover, we shall split $F^{r}$ defined in \eqref{eq_urkernel} %\Blue{from \eqref{eq_urkernel}; $$ F^{r}(r,s,\br{r},\br{s})=\frac{2}{\pi^{2}}\int_{0}^{\pi}\int_{0}^{\pi}\frac{\br{r}\br{s}\cos\br{\tht}(\br{s}-s\cos\br{\phi})}{[(r-\br{r})^{2}+(s-\br{s})^{2}+2r\br{r}(1-\cos\br{\tht})+2s\br{s}(1-\cos\br{\phi})]^{2}}d\br{\phi} d\br{\tht}, $$}
				into
				\begin{equation*}
					\begin{split}
						F^{r}(r,s,\br{r},\br{s})&=(I)+(II),\\
						(I)&:=-\frac{2}{\pi^{2}}\int_{0}^{\pi}\int_{0}^{\pi}\frac{\br{r}\br{s}\cos\br{\tht}(\br{s}-s)}{X_{--}%[(\br{r}-1)^{2}+(\br{s}-1)^{2}+2\br{r}(1-\cos\br{\tht})+2\br{s}(1-\cos\br{\phi})]
							^{2}}d\br{\phi}d\br{\tht},\quad 
						(II):=-\frac{2}{\pi^{2}}\int_{0}^{\pi}\int_{0}^{\pi}\frac{\br{r}s\br{s}\cos\br{\tht}(1-\cos\br{\phi})}{X_{--}%[(\br{r}-1)^{2}+(\br{s}-1)^{2}+2\br{r}(1-\cos\br{\tht})+2\br{s}(1-\cos\br{\phi})]
							^{2}}d\br{\phi}d\br{\tht}.
					\end{split}
				\end{equation*}
				We shall frequently use Lemma \ref{lem_estoff} to estimate $(I)$ and $(II)$.  More precisely, we have for instance \begin{equation*}
					\begin{split} 
						\int_{0}^{\pi}\frac{1}{X_{--}} d\br{\tht} \le 
						\frac{1}{2r \bar{r}}\int_{0}^{\pi}\frac{1}{ \frac{(r-\br{r})^{2}+(s-\br{s})^{2}}{2r\bar{r}} + (1-\cos\br{\tht})}d\br{\tht} \lesssim \frac{1}{(r\bar{r})^{\frac12}} \frac{1}{ ((r-\br{r})^{2}+(s-\br{s})^{2})^{\frac12} }. 
					\end{split}
				\end{equation*}
				
				\medskip
				\noindent 
				\textbf{Case 1.} $1/4\leq r\leq 3/4$.
				
				\medskip
				\noindent We split the domain $\Pi$ into
				\begin{equation*}
					\begin{split}
						A_{1}&:=\lbrace(\br{r},\br{s})\in\Pi : r-\frac{1}{8}\leq\br{r}\leq r+2, s-\frac{1}{8}\leq\br{s}\leq s+2\rbrace,\quad A_{2}:=\Pi\setminus A_{1},\\
						A_{21}&:=A_{2}\cap\lbrace(\br{r},\br{s})\in\Pi : \br{r}<r+2, \br{s}<s+2\rbrace,\quad A_{22}:=%A_{2}\cap
						\lbrace(\br{r},\br{s})\in\Pi : \br{r}>r+2, \br{s}<s+2\rbrace\subset A_{2},\\
						A_{23}&:=%A_{2}\cap
						\lbrace(\br{r},\br{s})\in\Pi : \br{r}<r+2, \br{s}>s+2\rbrace\subset A_{2},\quad A_{24}:=%A_{2}\cap
						\lbrace(\br{r},\br{s})\in\Pi : \br{r}>r+2, \br{s}>s+2\rbrace\subset A_{2}.
					\end{split}
				\end{equation*}

				\medskip
				\noindent 
				(i) $(\br{r},\br{s})\in A_{1}$.
				
				\medskip
				\noindent For $(I)$, we have
				\begin{equation*}
					\begin{split}
						|(I)|&\lesssim\int_{0}^{\pi}\int_{0}^{\pi}\frac{\br{r}\br{s}}{X_{--}%[(\br{r}-1)^{2}+(\br{s}-1)^{2}+2\br{r}(1-\cos\br{\tht})+2\br{s}(1-\cos\br{\phi})]
							^{3/2}}d\br{\phi}d\br{\tht}\lesssim\frac{\br{r}\br{s}^{1/2}}{s^{1/2}}\int_{0}^{\pi}\frac{1}{(r-\br{r})^{2}+(s-\br{s})^{2}+2r\br{r}(1-\cos\br{\tht})}d\br{\tht}\\
						&\lesssim\bigg(\frac{\br{r}\br{s}}{rs}\bigg)^{1/2}\frac{1}{[(r-\br{r})^{2}+(s-\br{s})^{2}]^{1/2}}\lesssim\frac{1}{[(r-\br{r})^{2}+(s-\br{s})^{2}]^{1/2}}.
					\end{split}
				\end{equation*}
				For $(II)$, we get
				\begin{equation*}
					\begin{split}
						|(II)|&\lesssim\int_{0}^{\pi}\int_{0}^{\pi}\frac{\br{r}[s\br{s}(1-\cos\br{\phi})]^{1/2}}{X_{--}%[(\br{r}-1)^{2}+(\br{s}-1)^{2}+2\br{r}(1-\cos\br{\tht})+2\br{s}(1-\cos\br{\phi})]
							^{3/2}}d\br{\phi}d\br{\tht}\lesssim%\frac{}{\br{s}^{1/2}}
						\int_{0}^{\pi}\frac{\br{r}}{(r-\br{r})^{2}+(s-\br{s})^{2}+2r\br{r}(1-\cos\br{\tht})}d\br{\tht}\\
						&\lesssim\frac{\br{r}^{1/2}}{r^{1/2}}\cdot\frac{1}{[(r-\br{r})^{2}+(s-\br{s})^{2}]^{1/2}}\lesssim\frac{1}{[(r-\br{r})^{2}+(s-\br{s})^{2}]^{1/2}}.
					\end{split}
				\end{equation*}
				Thus, we have
				\begin{equation*}%\label{eq_case1A1Fr}
					\begin{split}
						\bigg|\iint_{A_{1}}F^{r}(r,s,\br{r},\br{s})w(\br{r},\br{s})d\br{s}d\br{r}\bigg|&\lesssim\iint_{A_{1}}\frac{1}{[(r-\br{r})^{2}+(s-\br{s})^{2}]^{1/2}}|w(\br{r},\br{s})|d\br{s}d\br{r}\\
						&\leq\nrm{w}_{L^{1}(A_{1})}^{1/2}\nrm{w}_{L^{\ift}(A_{1})}^{1/2}\lesssim\nrm{rsw}_{L^{1}(A_{1})}^{1/2}\bigg\|\frac{w}{rs}\bigg\|_{L^{\ift}(A_{1})}^{1/2} {\leq[\nrm{r^{2}w}_{L^{1}(A_{1})}+\nrm{s^{2}w}_{L^{1}(A_{1})}]^{1/2}\bigg\|\frac{w}{rs}\bigg\|_{L^{\ift}(A_{1})}^{1/2}.}
						%\nrm{r^{2}s^{2}w}_{L^{1}(A_{1})}^{1/2}\nrm{w}_{L^{1}(A_{1})}^{1/4}\bigg\|\frac{w}{rs}\bigg\|_{L^{\ift}(A_{1})}^{1/2}.
					\end{split}
				\end{equation*}
				(ii) $(\br{r},\br{s})\in A_{2}$.
				
				\medskip
				\noindent Note that in this region, the term $ [(r-\br{r})^{2}+(s-\br{s})^{2}]^{1/2} $ has a constant lower bound: $[(r-\br{r})^{2}+(s-\br{s})^{2}]^{1/2}\gtrsim1.$
				\begin{comment}
					\begin{split}
						[(r-\br{r})^{2}+(s-\br{s})^{2}]^{1/2}\geq\begin{cases}
							\frac{1}{8} & \ \text{if}\ (\br{r},\br{s})\in A_{21},\\
							2 & \ \text{if}\ (\br{r},\br{s})\in A_{22}\cup A_{23},\\
							2\sqrt{2} & \ \text{if}\ (\br{r},\br{s})\in A_{24}.
						\end{cases}
					\end{split}
				\end{comment}
				Then using the triangular inequality, we have
				\begin{equation*}
					(\br{r}^{2}+\br{s}^{2})^{1/2}\leq[(r-\br{r})^{2}+(s-\br{s})^{2}]^{1/2}+(\underbrace{r^{2}+s^{2}}_{=1})^{1/2}\lesssim[(r-\br{r})^{2}+(s-\br{s})^{2}]^{1/2}.
				\end{equation*}
				Keeping this bound in mind, let us analyze the kernel $F^{r}$. We shall use the following notations:
				\begin{equation*}
					\begin{split}
						Y_{-}&:=\frac{X_{--}+X_{+-}}{2}=(r-\br{r})^{2}+(s-\br{s})^{2}+2r\br{r}+2s\br{s}(1-\cos\br{\phi}),\\
						Y_{+}&:=\frac{X_{-+}+X_{++}}{2}=(r-\br{r})^{2}+(s-\br{s})^{2}+2r\br{r}+2s\br{s}(1+\cos\br{\phi}).
					\end{split}
				\end{equation*}
				Observe that $F^{r}$ can be rewritten as
				\begin{equation*}
					\begin{split}
						F^{r}(r,s,\br{r},\br{s})&\simeq\int_{0}^{\pi}\int_{0}^{\pi}\frac{\br{r}\br{s}\cos\br{\tht}(\br{s}-s\cos\br{\phi})}{X_{--}^{2}}d\br{\phi}d\br{\tht}\simeq\int_{0}^{\frac{\pi}{2}}\int_{0}^{\pi}\frac{\br{r}^{2}\br{s}\cos^{2}\br{\tht}(\br{s}-s\cos\br{\phi})Y_{-}}{X_{--}^{2}X_{+-}^{2}}d\br{\phi}d\br{\tht}\\
						&=\int_{0}^{\frac{\pi}{2}}\int_{0}^{\frac{\pi}{2}}\br{r}^{2}\br{s}\cos^{2}\br{\tht}\bigg[\frac{(\br{s}-s\cos\br{\phi})Y_{-}}{X_{--}^{2}X_{+-}^{2}}+\frac{(\br{s}+s\cos\br{\phi})Y_{+}}{X_{-+}^{2}X_{++}^{2}}\bigg]d\br{\phi}d\br{\tht}.%\\
						%&=\int_{0}^{\frac{\pi}{2}}\int_{0}^{\frac{\pi}{2}}\frac{\br{r}^{2}\br{s}\cos^{2}\br{\tht}[(\br{s}-s\cos\br{\phi})Y_{-}X_{-+}^{2}X_{++}^{2}+(\br{s}+s\cos\br{\phi})Y_{+}X_{--}^{2}X_{+-}^{2}]}{X_{--}^{2}X_{+-}^{2}X_{-+}^{2}X_{++}^{2}}%\br{r}^{2}\br{s}\cos^{2}\br{\tht}\bigg[\frac{(\br{s}-\cos\br{\phi})Y_{-}}{X_{--}^{2}X_{+-}^{2}}+\frac{(\br{s}+\cos\br{\phi})Y_{+}}{X_{-+}^{2}X_{++}^{2}}\bigg]
						%d\br{\phi}d\br{\tht}
					\end{split}
				\end{equation*}
				When $(\br{r},\br{s})\in A_{21}$, we have
				\begin{equation*}
					\begin{split}
						|F^{r}(r,s,\br{r},\br{s})|&\lesssim\frac{1}{X_{--}^{2}X_{+-}}+\frac{1}{X_{-+}^{2}X_{++}}\leq\frac{1}{[(r-\br{r})^{2}+(s-\br{s})^{2}]^{3}}.
					\end{split}
				\end{equation*}
				For the case $(\br{r},\br{s})\in A_{22}$, we get
				\begin{equation*}
					\begin{split}
						|F^{r}(r,s,\br{r},\br{s})|&\lesssim\frac{\br{r}^{2}}{X_{--}^{2}X_{+-}}+\frac{\br{r}^{2}}{X_{-+}^{2}X_{++}}\leq\frac{1}{[(r-\br{r})^{2}+(s-\br{s})^{2}]^{2}}.
					\end{split}
				\end{equation*}
				If $(\br{r},\br{s})\in A_{23}$, then we have
				\begin{equation*}
					\begin{split}
						|F^{r}(r,s,\br{r},\br{s})|&\lesssim\frac{\br{s}^{2}}{X_{--}^{2}X_{+-}}+\frac{\br{s}^{2}}{X_{-+}^{2}X_{++}}\leq\frac{1}{[(r-\br{r})^{2}+(s-\br{s})^{2}]^{2}}.
					\end{split}
				\end{equation*}
				Lastly, when $(\br{r},\br{s})\in A_{24}$, we get
				\begin{equation*}
					\begin{split}
						|F^{r}(r,s,\br{r},\br{s})|&\lesssim\frac{\br{r}^{2}\br{s}^{2}}{X_{--}^{2}X_{+-}}+\frac{\br{r}^{2}\br{s}^{2}}{X_{-+}^{2}X_{++}}\leq\frac{1}{(r-\br{r})^{2}+(s-\br{s})^{2}}.
					\end{split}
				\end{equation*}
				In sum, in the region $A_{2}$, we obtain
				\begin{equation}\label{eq_FrA2est}
					\begin{split}
						|F^{r}(r,s,\br{r},\br{s})|&\lesssim%\frac{\br{r}^{2}\br{s}^{2}}{X_{--}^{2}X_{+-}}+\frac{\br{r}^{2}\br{s}^{2}}{X_{-+}^{2}X_{++}}\leq
						\frac{1}{(r-\br{r})^{2}+(s-\br{s})^{2}}.
					\end{split}
				\end{equation}
				Using this, we have by H\"older's inequality that 
				\begin{equation}\label{eq_case1A2Fr}
					\begin{split}
						\bigg|\iint_{A_{2}}F^{r}(r,s,\br{r},\br{s})w(\br{r},\br{s})d\br{s}d\br{r}\bigg|&\lesssim\iint_{A_{2}}\frac{1}{(\br{r}-1)^{2}+(\br{s}-1)^{2}}\cdot(\br{r}\br{s})^{1/2}\cdot|w(\br{r},\br{s})|^{1/2}%\cdot|w(\br{r},\br{s})|^{1/4}
						\cdot\frac{|w(\br{r},\br{s})|^{1/2}}{(\br{r}\br{s})^{1/2}}d\br{s}d\br{r}\\
						&\lesssim%\nrm{r^{2}s^{2}w}_{L^{1}(A_{21})}^{1/4}
						\nrm{rsw}_{L^{1}(A_{2})}^{1/2}\bigg\|\frac{w}{rs}\bigg\|_{L^{\ift}(A_{2})}^{1/2} {\leq[\nrm{r^{2}w}_{L^{1}(A_{2})}+\nrm{s^{2}w}_{L^{1}(A_{2})}]^{1/2}\bigg\|\frac{w}{rs}\bigg\|_{L^{\ift}(A_{2})}^{1/2}.}
					\end{split}
				\end{equation}
				
				\begin{comment} 
					\begin{note}
						This can be decomposed into
						\begin{equation}\label{eq_Frdecomp}
							\begin{split}
								F^{r}(r,s,\br{r},\br{s})\simeq\int_{0}^{\frac{\pi}{2}}\int_{0}^{\frac{\pi}{2}}\br{r}^{2}\br{s}\cos^{2}\br{\tht}\bigg[&-\frac{4s\br{s}\cos\br{\phi}(\br{s}-s\cos\br{\phi})}{X_{--}^{2}X_{+-}^{2}}-\frac{2s\cos\br{\phi}Y_{+}}{X_{--}^{2}X_{+-}^{2}}+\frac{4s\br{s}\cos\br{\phi}(\br{s}+s\cos\br{\phi})Y_{+}(X_{++}+X_{+-})}{X_{--}^{2}X_{+-}^{2}X_{++}^{2}}\\&+\frac{4s\br{s}\cos\br{\phi}(\br{s}+s\cos\br{\phi})Y_{+}(X_{-+}+X_{--})}{X_{--}^{2}X_{-+}^{2}X_{++}^{2}}+\Blue{\frac{2(\br{s}+s\cos\br{\phi})Y_{+}}{X_{-+}^{2}X_{++}^{2}}}\bigg]d\br{\phi}d\br{\tht}.
							\end{split}
						\end{equation}
						\Blue{I previously mistook the blue term above with $ (2s\br{s}\cos\br{\phi})/(X_{-+}^{2}X_{++}^{2}) $ due to an error in computation. Because of this confusion, I thought that in the region $A_{24}$, $F^{r}$ has the estimate $[(r-\br{r})^{2}+(s-\br{s})^{2}]^{-3/2}$. However, once the error is fixed, the estimate of $F^{r}$ in $A_{24}$ is simply $[(r-\br{r})^{2}+(s-\br{s})^{2}]^{-1}$. Nevertheless, this does not effect on the proof of Proposition \ref{prop_urLiftest}. Also, the decomposition \eqref{eq_Frdecomp} does not seem to be that useful.}
					\end{note}
				\end{comment} 
				
				\medskip
				\noindent 	
				\textbf{Case 2.} $0 \le r<1/4$.
				
				\medskip
				\noindent Assume for a moment that $0 < r < 1/4$. In this case, we split the domain $\Pi$ into
				\begin{equation*}
					\begin{split}
						A_{1}'&:=\lbrace(\br{r},\br{s})\in\Pi : 0\leq\br{r}\leq r+\frac{1}{4}, s-\frac{1}{4}\leq\br{s}\leq s+2\rbrace,\quad A_{2}':=\Pi\setminus A_{1}',\\
						%	A_{11}'&=\lbrace(\br{r},\br{s})\in\Pi : 0\leq\br{r}<r, s-\frac{1}{4}\leq\br{s}\leq s+2\rbrace\subset A_{1}',\quad A_{12}'=\lbrace(\br{r},\br{s})\in\Pi : r\leq\br{r}\leq r+\frac{1}{4}, s-\frac{1}{4}\leq\br{s}\leq s+2\rbrace\subset A_{1}',\\
						A_{21}'&:=%A_{2}\cap
						\lbrace(\br{r},\br{s})\in\Pi : 0\leq\br{r}\leq r+\frac{1}{4}, 0\leq\br{s}<s-\frac{1}{4}\rbrace\subset A_{2}',\quad A_{22}':=%A_{2}\cap
						\lbrace(\br{r},\br{s})\in\Pi : \br{r}>r+\frac{1}{4}, \br{s}\leq s+2\rbrace\subset A_{2}',\\
						A_{23}'&:=%A_{2}\cap
						\lbrace(\br{r},\br{s})\in\Pi : 0\leq\br{r}\leq r+\frac{1}{4}, \br{s}>s+2\rbrace\subset A_{2}',\quad A_{24}':%=A_{2}\cap
						\lbrace(\br{r},\br{s})\in\Pi : \br{r}>r+\frac{1}{4}, \br{s}>s+2\rbrace\subset A_{2}'.
					\end{split}
				\end{equation*}
				(i) $(\br{r},\br{s})\in A_{1}'$.
				
				\medskip
				\noindent  We note that \begin{equation*}
					\begin{split}
						\bar{r} \le \bar{r}^{1/2}(|r-\bar{r}|+r)^{1/2} \le \bar{r}^{1/2}|r-\bar{r}|^{1/2} + (\bar{r}r)^{1/2} \le \frac12\bar{r} + 2|r-\bar{r}| + (\bar{r}r)^{1/2} ,
					\end{split}
				\end{equation*} which gives $\bar{r} \lesssim |r-\bar{r}| + (\bar{r}r)^{1/2}$. Then for $(I)$, we have after applying Lemma \ref{lem_estoff} to integral in $\bar{\phi}$, 
				\begin{equation*}
					\begin{split}
						|(I)|&\lesssim\int_{0}^{\pi}\int_{0}^{\pi}\frac{\br{r}\br{s}}{X_{--}^{3/2}}d\br{\phi}d\br{\tht}
						\lesssim   \frac{\br{s}^{1/2}}{s^{1/2}}\int_{0}^{\pi}\frac{|r-\bar{r}| + (\bar{r}r)^{1/2}}{(r-\br{r})^{2}+(s-\br{s})^{2}+2r\br{r}(1-\cos\br{\tht})}d\br{\tht}.  
						%&\lesssim\bigg(\frac{\br{r}\br{s}}{rs}\bigg)^{1/2}\frac{1}{[(r-\br{r})^{2}+(s-\br{s})^{2}]^{1/2}}\lesssim\frac{1}{[(r-\br{r})^{2}+(s-\br{s})^{2}]^{1/2}}.
					\end{split}
				\end{equation*} Then, for the first term, we simply use $|r-\bar{r}| \le ((r-\br{r})^{2}+(s-\br{s})^{2}+2r\br{r}(1-\cos\br{\tht}))^{1/2}$ and for the second term, Lemma \ref{lem_estoff}. This gives \begin{equation*}
					\begin{split}
						|(I)| \lesssim \bigg(\frac{\br{s}}{s}\bigg)^{1/2}\frac{1}{[(r-\br{r})^{2}+(s-\br{s})^{2}]^{1/2}}\lesssim\frac{1}{[(r-\br{r})^{2}+(s-\br{s})^{2}]^{1/2}}.
					\end{split}
				\end{equation*}
				For $(II)$, we get after applying  $\bar{r} \lesssim |r-\bar{r}| + (\bar{r}r)^{1/2}$, 
				\begin{equation*}
					\begin{split}
						|(II)|&\lesssim\int_{0}^{\pi}\int_{0}^{\pi}\frac{\br{r}[s\br{s}(1-\cos\br{\phi})]^{1/2}}{X_{--}%[(\br{r}-1)^{2}+(\br{s}-1)^{2}+2\br{r}(1-\cos\br{\tht})+2\br{s}(1-\cos\br{\phi})]
							^{3/2}}d\br{\phi}d\br{\tht}\lesssim%\frac{\br{r}\br{s}^{1/2}}{s^{1/2}}
						\int_{0}^{\pi}\frac{ |r-\bar{r}| + (\bar{r}r)^{1/2}}{(r-\br{r})^{2}+(s-\br{s})^{2}+2r\br{r}(1-\cos\br{\tht})}d\br{\tht} 
						\lesssim  \frac{1}{[(r-\br{r})^{2}+(s-\br{s})^{2}]^{1/2}}
					\end{split}
				\end{equation*} similarly as in the case of $(I)$. %\Red{ }
				Hence, we obtain
				\begin{equation}\label{eq_case2A1Fr}
					\begin{split}
						\bigg|\iint_{A_{1}'}F^{r}(r,s,\br{r},\br{s})w(\br{r},\br{s})d\br{s}d\br{r}\bigg|&\lesssim\iint_{A_{1}'}\frac{|w(\br{r},\br{s})|}{[(r-\br{r})^{2}+(s-\br{s})^{2}]^{1/2}}d\br{s}d\br{r} \leq\nrm{w}_{L^{1}(A_{1}')}^{1/2}\nrm{w}_{L^{\ift}(A_{1}')}^{1/2}\lesssim\nrm{s^{2}w}_{L^{1}(A_{1}')}^{1/2}\bigg\|\frac{w}{rs}\bigg\|_{L^{\ift}(A_{1}')}^{1/2}.
						%\nrm{r^{2}s^{2}w}_{L^{1}(A_{1})}^{1/2}\nrm{w}_{L^{1}(A_{1})}^{1/4}\bigg\|\frac{w}{rs}\bigg\|_{L^{\ift}(A_{1})}^{1/2}.
					\end{split}
				\end{equation}
				(ii) $(\br{r},\br{s})\in A_{2}'$.
				
				\medskip
				\noindent Note that the term $ [(r-\br{r})^{2}+(s-\br{s})^{2}]^{1/2} $ has a constant lower bound: $[(r-\br{r})^{2}+(s-\br{s})^{2}]^{1/2}\gtrsim1.$ 
				\begin{comment}
					\begin{split}
						[(r-\br{r})^{2}+(s-\br{s})^{2}]^{1/2}\geq\begin{cases}
							\frac{1}{4} & \ \text{if}\ (\br{r},\br{s})\in A_{21}'\cup A_{22}',\\
							2 & \ \text{if}\ (\br{r},\br{s})\in A_{23}',\\
							\frac{\sqrt{65}}{4} & \ \text{if}\ (\br{r},\br{s})\in A_{24}'.
						\end{cases}
					\end{split}
				\end{comment}
				Then from the triangular inequality, we have
				\begin{equation*}
					(\br{r}^{2}+\br{s}^{2})^{1/2}\leq[(r-\br{r})^{2}+(s-\br{s})^{2}]^{1/2}+(\underbrace{r^{2}+s^{2}}_{=1})^{1/2}\lesssim[(r-\br{r})^{2}+(s-\br{s})^{2}]^{1/2}.
				\end{equation*}
				Using this, we repeat the process from \textbf{Case 1}-(ii) to get
				\begin{equation*}
					|F^{r}(r,s,\br{r},\br{s})|\lesssim\begin{cases}
						[(r-\br{r})^{2}+(s-\br{s})^{2}]^{-3} %\frac{1}{} 
						& \ \text{in} \ A_{21}',\\
						[(r-\br{r})^{2}+(s-\br{s})^{2}]^{-2}%\frac{1}{[(r-\br{r})^{2}+(s-\br{s})^{2}]^{2}} 
						& \ \text{in} \ A_{22}'\cup A_{23}',\\
						[(r-\br{r})^{2}+(s-\br{s})^{2}]^{-1}%\frac{1}{(r-\br{r})^{2}+(s-\br{s})^{2}} 
						& \ \text{in} \ A_{24}',\\
					\end{cases}
				\end{equation*}
				which gives us
				\begin{equation*}%\label{eq_case2A2Fr}
					\begin{split}
						\bigg|\iint_{A_{2}'}F^{r}(r,s,\br{r},\br{s})w(\br{r},\br{s})d\br{s}d\br{r}\bigg|&\lesssim\iint_{A_{2}'}\frac{1}{(\br{r}-1)^{2}+(\br{s}-1)^{2}}\cdot(\br{r}\br{s})^{1/2}\cdot|w(\br{r},\br{s})|^{1/2}%\cdot|w(\br{r},\br{s})|^{1/4}
						\cdot\frac{|w(\br{r},\br{s})|^{1/2}}{(\br{r}\br{s})^{1/2}}d\br{s}d\br{r}\\
						&\lesssim%\nrm{r^{2}s^{2}w}_{L^{1}(A_{21})}^{1/4}
						\nrm{rsw}_{L^{1}(A_{2}')}^{1/2}\bigg\|\frac{w}{rs}\bigg\|_{L^{\ift}(A_{2}')}^{1/2} {\leq[\nrm{r^{2}w}_{L^{1}(A_{2}')}+\nrm{s^{2}w}_{L^{1}(A_{2}')}]^{1/2}\bigg\|\frac{w}{rs}\bigg\|_{L^{\ift}(A_{2}')}^{1/2}.}
					\end{split}
				\end{equation*} Now when $r=0$, for any $(\br{r},\br{s})\in\Pi$, we simply have
				\begin{equation*}%\label{eq_case4Fr}
					\begin{split}
						F^{r}(0,1,\br{r},\br{s})%&=-\frac{1}{\pi^{2}}\int_{0}^{\pi}\int_{0}^{\pi}\frac{\br{r}\br{s}\cos\br{\tht}(\br{s}-\cos\br{\phi})}{[\br{r}^{2}+(\br{s}-1)^{2}+2\br{s}(1-\cos\br{\phi})]^{2}}d\br{\phi}d\br{\tht} 
						=-\frac{1}{\pi^{2}}\int_{0}^{\pi}\frac{\br{r}\br{s}(\br{s}-\cos\br{\phi})}{[\br{r}^{2}+(\br{s}-1)^{2}+2\br{s}(1-\cos\br{\phi})]^{2}}\bigg(\int_{0}^{\pi}\cos\br{\tht}d\br{\tht}\bigg)d\br{\phi}=0,
					\end{split}
				\end{equation*}
				which gives us $u^{r}(0,1)= 0.$ 
				
				\medskip
				\noindent 	
				\textbf{Case 3.} $3/4<r \le 1$.
				
				\medskip
				\noindent In this case, we split the domain $\Pi$ into
				\begin{equation*}
					\begin{split}
						A_{1}''&:=\lbrace(\br{r},\br{s})\in\Pi : r-\frac{1}{4}\leq\br{r}\leq r+2, 0\leq\br{s}\leq s+\frac{1}{4}\rbrace,\quad A_{2}'':=\Pi\setminus A_{1}'', \\
						A_{21}''&:=%A_{2}\cap
						\lbrace(\br{r},\br{s})\in\Pi : 0\leq\br{r}< r-\frac{1}{4}, 0\leq\br{s}\leq s+\frac{1}{4}\rbrace\subset A_{2}'',\quad A_{22}'':=%A_{2}\cap
						\lbrace(\br{r},\br{s})\in\Pi : \br{r}>r+2, 0\leq\br{s}\leq s+\frac{1}{4}\rbrace\subset A_{2}'',\\
						A_{23}''&:=%A_{2}\cap
						\lbrace(\br{r},\br{s})\in\Pi : 0\leq\br{r}\leq r+2, \br{s}>s+\frac{1}{4}\rbrace\subset A_{2}'',\quad A_{24}':=%A_{2}\cap
						\lbrace(\br{r},\br{s})\in\Pi : \br{r}>r+2, \br{s}>s+\frac{1}{4}\rbrace\subset A_{2}''.
					\end{split}
				\end{equation*}
				(i) $(\br{r},\br{s})\in A_{1}''$.
				
				\medskip
				\noindent We shall use the inequality $\bar{s} \lesssim (s \bar{s})^{1/2} + |s-\bar{s}|.$ For $(I)$, we have
				\begin{equation*}
					\begin{split}
						|(I)|&\lesssim\int_{0}^{\pi}\int_{0}^{\pi}\frac{\br{r}\br{s}}{X_{--}^{3/2}}d\br{\phi}d\br{\tht}\leq\int_{0}^{\pi}\int_{0}^{\pi}\frac{\br{r}((s \bar{s})^{1/2} + |s-\bar{s}|)}{X_{--}^{3/2}}d\br{\phi}d\br{\tht} \\
						& \lesssim 
						\int_{0}^{\pi}\frac{\br{r}}{(r-\br{r})^{2}+(s-\br{s})^{2}+2r\br{r}(1-\cos\br{\tht})}d\br{\tht} \lesssim\bigg(\frac{\br{r}}{r}\bigg)^{1/2}\frac{1}{[(r-\br{r})^{2}+(s-\br{s})^{2}]^{1/2}}\lesssim\frac{1}{[(r-\br{r})^{2}+(s-\br{s})^{2}]^{1/2}}.
						%&\lesssim\bigg(\frac{\br{r}\br{s}}{rs}\bigg)^{1/2}\frac{1}{[(r-\br{r})^{2}+(s-\br{s})^{2}]^{1/2}}\lesssim\frac{1}{[(r-\br{r})^{2}+(s-\br{s})^{2}]^{1/2}}.
					\end{split}
				\end{equation*}
				For $(II)$, we get
				\begin{equation*}
					\begin{split}
						|(II)|&\lesssim\int_{0}^{\pi}\int_{0}^{\pi}\frac{\br{r}[s\br{s}(1-\cos\br{\phi})]^{1/2}}{X_{--}%[(\br{r}-1)^{2}+(\br{s}-1)^{2}+2\br{r}(1-\cos\br{\tht})+2\br{s}(1-\cos\br{\phi})]
							^{3/2}}d\br{\phi}d\br{\tht}\lesssim%\frac{\br{r}\br{s}^{1/2}}{s^{1/2}}
						\int_{0}^{\pi}\frac{\br{r}}{(r-\br{r})^{2}+(s-\br{s})^{2}+2r\br{r}(1-\cos\br{\tht})}d\br{\tht}\\
						&\lesssim\bigg(\frac{\br{r}}{r}\bigg)^{1/2}\frac{1}{[(r-\br{r})^{2}+(s-\br{s})^{2}]^{1/2}}\lesssim\frac{1}{[(r-\br{r})^{2}+(s-\br{s})^{2}]^{1/2}}.
						%&\leq\int_{0}^{\pi}\frac{r^{1/2}\br{r}^{1/2}}{(r-\br{r})^{2}+(s-\br{s})^{2}+2r\br{r}(1-\cos\br{\tht})}d\br{\phi}d\br{\tht}%\leq\int_{0}^{\pi}\frac{|r-\br{r}|+r^{1/2}\br{r}^{1/2}}{(r-\br{r})^{2}+(s-\br{s})^{2}+2r\br{r}(1-\cos\br{\tht})}d\br{\phi}d\br{\tht}\\
						%&\leq\frac{r^{1/2}\br{r}^{1/2}\br{s}^{1/2}}{s^{1/2}}\int_{0}^{\pi}\frac{1}{(r-\br{r})^{2}+(s-\br{s})^{2}+2r\br{r}(1-\cos\br{\tht})}d\br{\phi}d\br{\tht}\\
						%\lesssim%\bigg(\frac{\br{s}}{s}\bigg)^{1/2}\frac{1}{[(r-\br{r})^{2}+(s-\br{s})^{2}]^{1/2}}\lesssim
						%\frac{1}{[(r-\br{r})^{2}+(s-\br{s})^{2}]^{1/2}}.
						%&\lesssim\bigg(\frac{\br{r}\br{s}}{rs}\bigg)^{1/2}\frac{1}{[(r-\br{r})^{2}+(s-\br{s})^{2}]^{1/2}}\lesssim\frac{1}{[(r-\br{r})^{2}+(s-\br{s})^{2}]^{1/2}}.
					\end{split}
				\end{equation*} Hence, we obtain
				\begin{equation}\label{eq_case3A1Fr}
					\begin{split}
						\bigg|\iint_{A_{1}''}F^{r}(r,s,\br{r},\br{s})w(\br{r},\br{s})d\br{s}d\br{r}\bigg|&\lesssim\iint_{A_{1}''}\frac{|w(\br{r},\br{s})|}{[(r-\br{r})^{2}+(s-\br{s})^{2}]^{1/2}}d\br{s}d\br{r}\leq\nrm{w}_{L^{1}(A_{1}'')}^{1/2}\nrm{w}_{L^{\ift}(A_{1}'')}^{1/2}\lesssim\nrm{r^{2}w}_{L^{1}(A_{1}'')}^{1/2}\bigg\|\frac{w}{rs}\bigg\|_{L^{\ift}(A_{1}'')}^{1/2}.
						%\nrm{r^{2}s^{2}w}_{L^{1}(A_{1})}^{1/2}\nrm{w}_{L^{1}(A_{1})}^{1/4}\bigg\|\frac{w}{rs}\bigg\|_{L^{\ift}(A_{1})}^{1/2}.
					\end{split}
				\end{equation}
				(ii) $(\br{r},\br{s})\in A_{2}''$.
				
				\medskip
				\noindent Note that the term $ [(r-\br{r})^{2}+(s-\br{s})^{2}]^{1/2} $ has a constant lower bound: $	[(r-\br{r})^{2}+(s-\br{s})^{2}]^{1/2}\gtrsim1.$
				Then from the triangular inequality, we have
				\begin{equation*}
					(\br{r}^{2}+\br{s}^{2})^{1/2}\leq[(r-\br{r})^{2}+(s-\br{s})^{2}]^{1/2}+(\underbrace{r^{2}+s^{2}}_{=1})^{1/2}\lesssim[(r-\br{r})^{2}+(s-\br{s})^{2}]^{1/2}.
				\end{equation*}
				Using this, we repeat the process from \textbf{Case 1}-(ii) to get
				\begin{equation*}
					|F^{r}(r,s,\br{r},\br{s})|\lesssim\begin{cases}
						[(r-\br{r})^{2}+(s-\br{s})^{2}]^{-3}%\frac{1}{[(r-\br{r})^{2}+(s-\br{s})^{2}]^{3}} 
						& \ \text{in} \ A_{21}'',\\
						[(r-\br{r})^{2}+(s-\br{s})^{2}]^{-2}%\frac{1}{[(r-\br{r})^{2}+(s-\br{s})^{2}]^{2}} 
						& \ \text{in} \ A_{22}''\cup A_{23}'',\\
						[(r-\br{r})^{2}+(s-\br{s})^{2}]^{-1}%\frac{1}{(r-\br{r})^{2}+(s-\br{s})^{2}} 
						& \ \text{in} \ A_{24}'',\\
					\end{cases}
				\end{equation*}
				which gives us
				\begin{equation}\label{eq_case3A2Fr}
					\begin{split}
						\bigg|\iint_{A_{2}''}F^{r}(r,s,\br{r},\br{s})w(\br{r},\br{s})d\br{s}d\br{r}\bigg|&\lesssim\iint_{A_{2}''}\frac{1}{(\br{r}-1)^{2}+(\br{s}-1)^{2}}\cdot(\br{r}\br{s})^{1/2}\cdot|w(\br{r},\br{s})|^{1/2}%\cdot|w(\br{r},\br{s})|^{1/4}
						\cdot\frac{|w(\br{r},\br{s})|^{1/2}}{(\br{r}\br{s})^{1/2}}d\br{s}d\br{r}\\
						&\lesssim%\nrm{r^{2}s^{2}w}_{L^{1}(A_{21})}^{1/4}
						\nrm{rsw}_{L^{1}(A_{2}'')}^{1/2}\bigg\|\frac{w}{rs}\bigg\|_{L^{\ift}(A_{2}'')}^{1/2} {\leq[\nrm{r^{2}w}_{L^{1}(A_{2}'')}+\nrm{s^{2}w}_{L^{1}(A_{2}'')}]^{1/2}\bigg\|\frac{w}{rs}\bigg\|_{L^{\ift}(A_{2}'')}^{1/2}.}
					\end{split}
				\end{equation}
				Note that when $r = 1$ (which implies $s = 0$), the same argument goes through ($(II)=0$ in this case).

				\medskip
				
				\noindent 		Gathering \textbf{Case 1}--\textbf{Case 3}, we have shown that
				\begin{equation*}
					\sup_{\substack{(r,s)\in\Pi \\ r^{2}+s^{2}=1}}|u^{r}(r,s)|\lesssim[%\nrm{rsw}_{L^{1}(\Pi)}+
					\nrm{r^{2}w}_{L^{1}(\Pi)}+\nrm{s^{2}w}_{L^{1}(\Pi)}]^{1/2}
					%[\nrm{rsw}_{L^{1}(\Pi)}^{1/2}+\nrm{r^{2}w}_{L^{1}(\Pi)}^{1/2}+\nrm{s^{2}w}_{L^{1}(\Pi)}^{1/2}]
					\cdot\bigg\|\frac{w}{rs}\bigg\|_{L^{\ift}(\Pi)}^{1/2}.
				\end{equation*}
				For $u^{s}$, since the kernel $F^{s}$ of $u^{s}$ has the similar form to the kernel $F^{r}$ of $u^{r}$ (i.e., $F^{s}(r,s,\br{r},\br{s})=F^{r}(s,r,\br{s},\br{r})$), we can repeat the same procedure to obtain the claimed bound for $u^{s}$. This finishes the proof. 
			\end{proof}
			
			\subsection{Proof of global regularity}\label{subsec_pfglob}
			Finally, we finish the proof of %this paper by proving 
			Theorem \ref{thm_globreg}.
			
			\begin{proof}[Proof of Theorem \ref{thm_globreg}]
				We let $t\geq0$ and assume that $ w_{0} $ satisfies conditions from Theorem \ref{thm_localreg} and 
				$$ (1+r+s)%(1+r)(1+s)
				\frac{w_{0}}{rs}\in L^{\ift}%(L^{1}\cap L^{\ift})
				(\bbR^{4}),\quad (1+r^{2}+s^{2})\frac{w_{0}}{rs}\in L^{1}(\bbR^{4}). $$
				%$ \Green{} $, %$(1+r^{-1})(1+s^{-1})w_{0}\in (L^{1}\cap L^{\ift})(\bbR^{4})$ 
				%$ \Green{.} $  %$(rs^{-1}+r^{-1}s)w_{0}\in L^{1}(\bbR^{4})$. 
				First, let us prove the estimate \eqref{eq_wtgrowth}. 
				To prove this, %Then 
				by Lemma \ref{lem_wtest}, it suffices to show that the function $ L(t) $ %$L(t)=1+\int_{0}^{t}[\nrm{u^{r}(\tau)}_{L^{\ift}(\bbR^{4})}+\nrm{u^{s}(\tau)}_{L^{\ift}(\bbR^{4})}]d\tau$ 
				from \eqref{eq_Ltdef} satisfies
				\begin{equation}\label{eq_Ltexp}
					L(t)\lesssim_{w_{0}} e^{ct},
				\end{equation}
				for some constant $c>0$ having dependence on the initial data $w_{0}$.
				
				\medskip
				\noindent 
				To prove this, note that using Proposition \ref{prop_urLiftest}, its time derivative is estimated as
				\begin{equation*}
					\begin{split}
						L'(t)&=\nrm{u^{r}(t)}_{L^{\ift}(\bbR^{4})}+\nrm{u^{s}(t)}_{L^{\ift}(\bbR^{4})}\lesssim\bigg[%\nrm{w(t)}_{L^{1}(\bbR^{4})}+
						\bigg\|\frac{rw(t)}{s}\bigg\|_{L^{1}(\bbR^{4})}+\bigg\|\frac{sw(t)}{r}\bigg\|_{L^{1}(\bbR^{4})}\bigg]^{1/2}
						%\bigg[\nrm{w(t)}_{L^{1}(\bbR^{4})}^{1/2}+\bigg\|\frac{rw(t)}{s}\bigg\|_{L^{1}(\bbR^{4})}^{1/2}+\bigg\|\frac{sw(t)}{r}\bigg\|_{L^{1}(\bbR^{4})}^{1/2}\bigg]
						\cdot\bigg\|\frac{w_{0}}{rs}\bigg\|_{L^{\ift}(\bbR^{4})}^{1/2},
					\end{split}
				\end{equation*}
				where we used the conservation of $\nrm{\frac{w}{rs}}_{L^{\ift}(\bbR^{4})}$ in time. Then using estimates from Lemma \ref{lem_wL1}, we obtain
				\begin{equation}\label{eq_Lprimeest}
					\begin{split}
						L'(t)&\lesssim\bigg[%\nrm{w_{0}}_{L^{1}(\bbR^{4})}+
						\bigg\|\frac{sw_{0}}{r}\bigg\|_{L^{1}(\bbR^{4})}+\bigg\|\frac{rw_{0}}{s}\bigg\|_{L^{1}(\bbR^{4})}
						%+\bigg\|\frac{w_{0}}{r}\bigg\|_{L^{1}(\bbR^{4})}+\bigg\|\frac{w_{0}}{s}\bigg\|_{L^{1}(\bbR^{4})}
						+\bigg\|\frac{w_{0}}{rs}\bigg\|_{L^{1}(\bbR^{4})}\bigg]^{1/2}\cdot L(t)\cdot\bigg\|\frac{w_{0}}{rs}\bigg\|_{L^{\ift}(\bbR^{4})}^{1/2}
						\lesssim_{w_{0}}L(t).
					\end{split}
				\end{equation}
				Thus, %Therefore, 
				we have shown \eqref{eq_Ltexp}. Then combined with the estimate \eqref{eq_wtest} from Lemma \ref{lem_wtest}, we have the estimate of the term $ \nrm{w(t)}_{L^{\ift}(\bbR^{4})} $:\begin{equation}\label{eq_wtexp}
					\nrm{w(t)}_{L^{\ift}(\bbR^{4})}\lesssim_{w_{0}}e^{ct}.
				\end{equation}
				
				%\medskip
				
				\noindent Now it is left to show that the $ L^{4,1} $-norm of the term $ (1+r+s+rs)w(t)/(rs) $ is bounded for any $ t\geq0 $. To begin with, the boundedness of the term $ \nrm{w(t)/(rs)}_{L^{4,1}(\bbR^{4})} $ is immediate from the conservation of $ w(t)/(rs) $ along the flow map:
				\begin{equation}\label{eq_woverrsL4,1}
					\bigg\|\frac{w(t)}{rs}\bigg\|_{L^{4,1}(\bbR^{4})}=\bigg\|\frac{w_{0}}{rs}\bigg\|_{L^{4,1}(\bbR^{4})}.
				\end{equation}
				Next, note that by the estimate \eqref{eq_Ltexp} above, estimates \eqref{eq_rwoversL1est} and \eqref{eq_swoverrL1est} from Lemma \ref{lem_wL1}, and the estimate \eqref{eq_urLiftest} from Proposition \ref{prop_urLiftest}, we have exponential estimates of terms $ \nrm{u^{r}(t)}_{L^{\ift}(\bbR^{4})} $ and $ \nrm{u^{s}(t)}_{L^{\ift}(\bbR^{4})} $ as well:
				\begin{equation}\label{eq_urusexp}
					\nrm{u^{r}(t)}_{L^{\ift}(\bbR^{4})}, \nrm{u^{s}(t)}_{L^{\ift}(\bbR^{4})}\lesssim_{w_{0}}e^{ct}.
				\end{equation}
				Now observe that the right-hand sides of equations \eqref{eq_woverr} and \eqref{eq_wovers} for $ w/r $ and $ w/s $ are $ u^{s}w/(rs) $ and $ u^{r}w/(rs) $, respectively. Hence, by combining the estimate \eqref{eq_urusexp} and the conservation \eqref{eq_woverrsL4,1}, we have exponential bounds of terms $ \nrm{w(t)/r}_{L^{4,1}(\bbR^{4})} $ and $ \nrm{w(t)/s}_{L^{4,1}(\bbR^{4})} $:
				\begin{equation}\label{eq_woverr,woversL4,1}
					\bigg\|\frac{w(t)}{r}\bigg\|_{L^{4,1}(\bbR^{4})}, \bigg\|\frac{w(t)}{s}\bigg\|_{L^{4,1}(\bbR^{4})}\lesssim_{w_{0}}e^{ct}.
				\end{equation}
				Finally, plugging in the exponential estimate \eqref{eq_woverr,woversL4,1} in the inequality \eqref{eq_wL4,1}, we obtain an double-exponential estimate of the term $ \nrm{w(t)}_{L^{4,1}(\bbR^{4})} $:
				\begin{equation}\label{eq_wL4,1doubexp}
					\nrm{w(t)}_{L^{4,1}(\bbR^{4})}\lesssim_{w_{0}}e^{e^{ct}}.
				\end{equation}
				This finishes the proof.
				%and the proof is done.
			\end{proof}

			\subsection*{Acknowledgments}
			
			KC has been supported by the National Research Foundation of Korea(NRF) grant funded by the Korea government(MSIT)(grant No. 2022R1A4A1032094, RS-2023-00274499). IJ and DL have been supported by the Samsung Science and Technology Foundation under Project Number SSTF-BA2002-04.
			
			\appendix
			
			\section*{Appendix}

			\section{Vorticity computation}\label{seca:vorticity}
			
			In this section, we assume that the solution of the 4D Euler equations satisfy the bi-rotational symmetry, and compute the vorticity tensor $\omg$ in terms of $w, u^{\tht},$ and $u^{\phi}$. The following formulas will be used: we have \begin{equation}  \label{eq:r-to-x} 
				\begin{aligned}
					\rd_{x_1} &= \cos\tht\rd_{r} - \sin\tht \frac{\rd_{\tht}}{r}, \quad \rd_{x_2} = \sin\tht\rd_{r} + \cos\tht \frac{\rd_{\tht}}{r}, \quad  \rd_{x_3} = \cos\phi\rd_{s} - \sin\phi \frac{\rd_{\phi}}{s}, \quad \rd_{x_4} = \sin\phi\rd_{s} + \cos\phi \frac{\rd_{\phi}}{s},
				\end{aligned} 
			\end{equation} and \begin{equation}\label{eq:u-to-v}
				\begin{split}
					u^1 & = \cos\tht u^{r} - \sin\tht u^{\tht}, \quad u^2 = \sin\tht u^{r} + \cos\tht u^{\tht}, \quad u^3 = \cos\phi u^{s} - \sin\phi u^{\phi}, \quad u^4  = \sin\phi u^{s} + \cos\phi u^{\phi}.
				\end{split}
			\end{equation} Then, we can compute that \begin{equation*}
				\begin{split}
					\omg^{1,2} &= (\sin\tht \rd_{r} + \cos\tht \frac{\rd_{\tht}}{r})(\cos\tht u^{r} - \sin\tht u^{\tht}) - (\cos\tht \rd_{r}- \sin\tht \frac{\rd_{\tht}}{r})(\sin\tht u^{r} + \cos\tht u^{\tht}) =-\rd_{r} u^{\tht} - \frac{u^{\tht}}{r}
				\end{split}
			\end{equation*} and similarly \begin{equation*}
				\begin{split}
					\omg^{3,4} = -\rd_{s} u^{\phi} - \frac{u^{\phi}}{s}.
				\end{split}
			\end{equation*} Next, \begin{equation}\label{eq_vorttensor}
				\begin{split}
					\omg^{1,3} & = -\cos\tht \cos\phi w - \sin\tht \cos\tht \rd_{s}u^{\tht} + \cos\tht \sin\phi \rd_{r}u^{\phi},\\
					\omg^{1,4} & = -\cos\tht \sin\phi w - \sin\tht \sin\phi \rd_{s}u^{\tht} - \cos\tht \cos\phi \rd_{r}u^{\phi},\\
					\omg^{2,3} & = -\sin\tht \cos\phi w + \cos\tht \cos\phi \rd_{s}u^{\tht} + \sin\tht \sin\phi \rd_{r}u^{\phi},\\
					\omg^{2,4} & = -\sin\tht \sin\phi w + \cos\tht \sin\phi \rd_{s}u^{\tht} - \sin\tht \cos\phi \rd_{r}u^{\phi}. 
				\end{split}
			\end{equation} One sees that in the no-swirl case, every component of the vorticity tensor is a multiple of $w$ with coefficient depending only on $\tht$ and $\phi$. In particular when $w \in L^p$ (or in $L^{p,q}$), we have automatically $\omg \in L^p$ (in $L^{p,q}$, respectively).
			\medskip
			
			\noindent Now we provide a proof of the fact that quantities $ w/r $, $ w/s $, and $ w/(rs) $ are uniformly bounded in $ \bbR^{4} $ if their corresponding velocity $ u $ with bi-rotational symmetry without swirl is sufficiently regular. Recall that the initial data $u_{0}$ mentioned in Remark \ref{rmk_globaldata} is in $H^{s}(\bbR^{4})$ for some $s>5$. Then by the Sobolev embedding theorem, $u_{0}$ is in $C^{3}(\bbR^{4})$, that is, its derivatives are continuous and uniformly bounded in $\bbR^{4}$ up to order 3. This is the type of velocity that is considered in the proposition below. \begin{prop}\label{prop_Liftwoverrs}
					Let $ u\in C^{3}(\bbR^{4}) $ satisfy the bi-rotational symmetry \eqref{eq_bi-rotational} and the no-swirl condition \eqref{eq_noswirl}. Then we have
					\begin{equation}\label{eq_Liftwoverrs}
						(1+r+s)\frac{w}{rs}\in L^{\ift}(\bbR^{4}).
					\end{equation}
			\end{prop}
			\begin{proof}
				The uniformly boundedness of $ w/r $ and $ w/s $ are immediate from the regularity and the bi-rotational symmetry of $ u $, so it suffices to prove $ w/(rs)\in L^{\ift}(\bbR^{4}) $.
				\medskip
				
				\noindent First, note from \eqref{eq_vorttensor} that the bi-rotational flow without swirl satisfies
				\begin{equation}\label{eq_omg13w}
					\omg^{1,3}(x_{1},x_{2},x_{3},x_{4})=-\cos\tht\cos\phi w(r,s),\quad x\in\bbR^{4}.
				\end{equation}
				Then note that if $x_{1}=0$, then we have either $\cos\tht=0$ or $r=0$, where $r=0$ implies $w(r,s)=0$ because of the bi-rotational symmetry of $u$ and the regularity of $w$. Thus, we have
				$$ \omg^{1,3}(0,x_{2},x_{3},x_{4})=0,\quad x_{2},x_{3},x_{4}\in \bbR. $$
				Likewise, we also have
				$$ \omg^{1,3}(x_{1},x_{2},0,x_{4})=0,\quad x_{1},x_{2},x_{4}\in \bbR. $$
				In addition, note that from $u\in C^{3}(\bbR^{4})$, we get $\omg^{1,3}\in C^{2}(\bbR^{4})$. Then for each $x_{2},x_{4}\in \bbR$, 
				%Then due to the $C^{2}$-regularity of $\omg^{1,3}$ in $\bbR^{4}$, which is %from $u\in C^{3}(\bbR^{4})$, %$\omg^{1,3}\in C^{2}(\bbR^{4})$, 
				the term
				\begin{equation*}
					\begin{split}
						\frac{\omg^{1,3}(x_{1},x_{2},x_{3},x_{4})}{x_{1}x_{3}}&=\frac{1}{x_{3}}\cdot\bigg(\frac{\omg^{1,3}(x_{1},x_{2},x_{3},x_{4})-\omg^{1,3}(0,x_{2},x_{3},x_{4})}{x_{1}}-\frac{\omg^{1,3}(x_{1},x_{2},0,x_{4})-\omg^{1,3}(0,x_{2},0,x_{4})}{x_{1}}\bigg)\\
						&=\frac{1}{x_{1}}\cdot\bigg(\frac{\omg^{1,3}(x_{1},x_{2},x_{3},x_{4})-\omg^{1,3}(x_{1},x_{2},0,x_{4})}{x_{3}}-\frac{\omg^{1,3}(0,x_{2},x_{3},x_{4})-\omg^{1,3}(0,x_{2},0,x_{4})}{x_{3}}\bigg)
					\end{split}
				\end{equation*}
				converges to $\rd_{x_{3}x_{1}}\omg^{1,3}(0,x_{2},0,x_{4})=\rd_{x_{1}x_{3}}\omg^{1,3}(0,x_{2},0,x_{4})<\ift$ as $x_{1}\to0$ and $x_{3}\to0$. From this, 
				we have $\omg^{1,3}/(x_{1}x_{3})\in L^{\ift}(\bbR^{4})$. Finally, by the relation
				$$ \frac{\omg^{1,3}(x_{1},x_{2},x_{3},x_{4})}{x_{1}x_{3}}=-\frac{w(r,s)}{rs},\quad x_{1}, x_{3}\neq0, $$
				we obtain
				$w/(rs)\in L^{\ift}(\bbR^{4})$.
		\end{proof}

			\bibliographystyle{plain}
			\bibliography{biblography_CJL_240219}
			
		\end{document}